\documentclass[12pt]{amsart}
\usepackage{amssymb}
\usepackage[shortlabels]{enumitem}
\usepackage[all]{xy}
\usepackage{tikz}
\usetikzlibrary{arrows}
\usepackage{xcolor}
\usepackage[margin=1in]{geometry} 
\usepackage{hyperref}
\usepackage{mathrsfs}
\usepackage{eucal}
\usepackage{contour}
\usepackage{shuffle}
\usepackage[normalem]{ulem}
\hypersetup{bookmarksdepth=2}
\hypersetup{colorlinks=true}
\hypersetup{linkcolor=blue}
\hypersetup{citecolor=blue}
\hypersetup{urlcolor=blue}

\makeatletter
\@addtoreset{equation}{section}
\makeatother

\numberwithin{equation}{section}
\newtheorem{theorem}[equation]{Theorem}
\newtheorem{proposition}[equation]{Proposition}
\newtheorem{lemma}[equation]{Lemma}
\newtheorem{corollary}[equation]{Corollary}

\newtheorem{maintheorem}{Theorem}

\theoremstyle{definition}
\newtheorem{remark}[equation]{Remark}
\newtheorem{question}[equation]{Question}
\newtheorem{example}[equation]{Example}
\newtheorem{definition}[equation]{Definition}

\newcommand{\cC}{\mathcal{C}}

\newcommand{\sC}{\mathscr{C}}

\newcommand{\cD}{\mathcal{D}}

\newcommand{\bG}{\mathbf{G}}

\newcommand{\rG}{\mathrm{G}}

\newcommand{\rK}{\mathrm{K}}

\newcommand{\rM}{\mathrm{M}}

\newcommand{\bN}{\mathbf{N}}

\newcommand{\bQ}{\mathbf{Q}}

\newcommand{\bR}{\mathbf{R}}

\newcommand{\bS}{\mathbf{S}}

\newcommand{\fS}{\mathfrak{S}}

\newcommand{\bZ}{\mathbf{Z}}

\newcommand{\GG}{\mathbb{G}}
\newcommand{\HH}{\mathbb{H}}

\newcommand{\arxiv}[1]{\href{http://arxiv.org/abs/#1}{{\tiny\tt arXiv:#1}}}

\newcommand{\DOI}[1]{\href{http://doi.org/#1}{\color{purple}{\tiny\tt DOI:#1}}}
\newcommand{\defn}[1]{\emph{#1}}

\contourlength{1pt}

\contourlength{0.8pt}
\newcommand{\myuline}[1]{%
  \uline{\phantom{#1}}%
  \llap{\contour{white}{#1}}%
}
\DeclareMathOperator{\uRep}{\text{\myuline{\rm Rep}}}
\DeclareMathOperator{\uPerm}{\ul{Perm}}

\let\ul\underline

\renewcommand{\phi}{\varphi}
\DeclareMathOperator{\tr}{tr}
\DeclareMathOperator{\ptr}{ptr}

\DeclareMathOperator{\Eq}{Eq}
\DeclareMathOperator{\End}{End}
\DeclareMathOperator{\Aut}{Aut}
\DeclareMathOperator{\Isom}{Isom}
\DeclareMathOperator{\Hom}{Hom}

\DeclareMathOperator{\Rep}{Rep}
\DeclareMathOperator{\ev}{ev}
\DeclareMathOperator{\cv}{cv}
\DeclareMathOperator{\Sym}{Sym}
\DeclareMathOperator{\Et}{Et}
\DeclareMathOperator{\im}{im}

\DeclareMathOperator{\rev}{rev}

\newcommand{\id}{\mathrm{id}}

\renewcommand{\Vec}{\mathrm{Vec}}
\newcommand{\sVec}{\mathrm{sVec}}

\newcommand{\op}{\mathrm{op}}
\newcommand{\lw}{{\textstyle\bigwedge}}

\newcommand{\bone}{\mathbf{1}}

\newcommand{\bb}{{\bullet}}
\newcommand{\ww}{{\circ}}

\title[Characterization of the Delannoy category by Adams operations]{A characterization of the Delannoy category \\ by Adams operations}

\author{Andrew Snowden}
\address{Department of Mathematics, University of Michigan, Ann Arbor, MI, USA}
\email{\href{mailto:asnowden@umich.edu}{asnowden@umich.edu}}
\urladdr{\url{http://www-personal.umich.edu/~asnowden/}}
\thanks{AS was supported by NSF grant DMS-2301871.}

\author{Noah Snyder}
\address{Department of Mathematics, Indiana University, Bloomington, IN, USA}
\email{\href{mailto:nsnyder1@indiana.edu}{nsnyder1@indiana.edu}}
\urladdr{\url{https://nsnyder1.pages.iu.edu/}}
\thanks{NS was supported by NSF grant DMS-2453032 and Simons grant MPS-TSM-00007608.}

\date{November 7, 2025}

\begin{document}

\begin{abstract}
In recent joint work with Harman, we studied a pre-Tannakian category called the Delannoy category, and showed that it had numerous special properties. One of these is that the Adams operations on its Grothendieck group are trivial. In this paper, we prove three theorems inspired by this. Theorem~A states that the Delannoy category is the unique semi-simple pre-Tannakian category having a generator that is fixed by the second Adams operation and whose exterior powers are simple. Theorem~B states the Delannoy category is uniquely determined by its Grothendieck semi-ring (among semi-simple pre-Tannakian categories). This is reminiscent of Kazhdan--Wenzl's recognition theorem for quantum $\mathfrak{gl}_n$, and many subsequent results. Finally, Theorem~C establishes some properties of pre-Tannakian categories where the second Adams operation fixes a generator.
\end{abstract}

\maketitle
\tableofcontents

\section{Introduction}

In recent joint work with Harman \cite{line}, we studied a pre-Tannakian tensor category called the \defn{Delannoy category}. Of its numerous special properties, one is that the Adams operations on its Grothendieck group are trivial. The Delannoy category was the first known non-trivial semi-simple category with this property, and remains (essentially) the only known example. In view of this, two natural questions are: (a) To what extent does this property characterize the Delannoy category? (b) What, if anything, can one say about a general category with this property? In this paper, we give a precise answer to the first question, and make progress on the second.

\subsection{Adams operations} \label{ss:adams}

Before stating our results, we provide some brief background on Adams operations. Let $\cC$ be a semi-simple pre-Tannakian tensor category over an algebraically closed field $k$, and let $\rK(\cC)$ denotes its Grothendieck group; see \S \ref{ss:tencat} for tensor category terminology. For each prime number $p$ different from the characteristic, there is a canonical ring homomorphism
\begin{displaymath}
\psi^p \colon \rK(\cC) \to \rK(\cC)
\end{displaymath}
called the \defn{Adams operation}\footnote{There is some subtlety in the construction when $k$ has positive characteristic and $p$ exceeds the characteristic. The method of \cite{Benson} can be used. This case will not play a major role for us.}. If $\cC$ is the category of complex representations of a finite group $G$, then the Adams operations have a simple description in terms of characters: $\psi^p V$ is the virtual representation with $\chi_{\psi^p V}(g) = \chi_V(g^p)$. For the purposes of this paper, the operator $\psi^2$ will be the most important. It is given explicitly by the formula
\begin{displaymath}
\psi^2([X]) = [\Sym^2{X}] - [\lw^2{X}].
\end{displaymath}

As mentioned, it is quite rare for all of the Adams operations to be trivial on $\rK(\cC)$. This is the case for the Delannoy category and its Deligne tensor powers, and these are the only known semi-simple examples. The Delannoy category is associated to the oligomorphic group $\GG=\Aut(\bR,<)$, and in \cite{homoperm} we showed that there is a semi-simple pre-Tannakian category associated to the wreath product $\GG \wr \GG$. It is plausible that this category has the property as well, but this has not been proved yet; of course, one could also consider iterated wreath products. Building on work of Czenky \cite{Czenky}, we show that if $\cC$ has moderate growth and $\psi^2$ is trivial on $\rK(\cC)$ (and $k$ does not have characteristic~2) then $\cC$ is the category of vector spaces (see Corollary~\ref{cor:moderate}). Thus any interesting example with trivial $\psi^2$ must have super-exponential growth.

\subsection{Characterization of Delannoy}

We assume\footnote{With minor additional work, our results can be extended to characteristics 5 and 7.} from now on that our coefficient field does not have characteristic 2, 3, 5, or 7. We let $\cD$ denote the Delannoy category, which is a semi-simple pre-Tannakian category. The Delannoy category has two ``smallest'' non-trivial simple objects $L_{\bb}$ and $L_{\ww}$ that are dual to each other. All exterior powers of these two simples remain simple. We review additional background on $\cD$ in \S \ref{ss:delannoy}, but for our main results, this is all one needs to know. Our first main theorem is a characterization of $\cD$:

\begin{maintheorem} \label{thm:A}
Let $\cC$ be a semi-simple pre-Tannakian category with a generating object $X$ satisfying the following conditions:
\begin{enumerate}[(i)]
\item The class $[X]$ in $\rK(\cC)$ is fixed by $\psi^2$.
\item All exterior powers of $X$ are simple.
\end{enumerate}
Then $\cC$ is equivalent to the Delannoy category. Precisely, there is an equivalence of tensor categories $\Phi \colon \cD \to \cC$ satisfying $\Phi(L_{\bb})=X$, which is unique up to isomorphism.
\end{maintheorem}

Let $\rK_+(\cC)$ denote the Grothendieck semi-ring of $\cC$, i.e., the positive cone in $\rK(\cC)$. Using similar methods to Theorem~\ref{thm:A}, we show that the Delannoy category is uniquely characterized by its Grothendieck semi-ring, in the following sense:

\begin{maintheorem} \label{thm:B}
If $\cC$ is a semi-simple pre-Tannakian category with a semi-ring isomorphism $i \colon \rK_+(\cD) \to \rK_+(\cC)$ then $\cC$ is equivalent to the Delannoy category $\cD$.
\end{maintheorem}

In fact, we show that the semi-ring isomorphism $i$ is induced by a (possibly contravariant) equivalence $\cD \to \cC$. In the course of the proof, we show that the auto-equivalence group of $\cD$ has order two (or four if contravariant functors are alowed). See \S \ref{ss:thmB} for details.

\begin{remark} \label{rem:recognition}
There is a long history of results like Theorem \ref{thm:B} saying that various monoidal categories of interest are characterized by their Grothendieck ring. The first and most famous of these is Kazhdan--Wenzl's recognition theorem for the category of representations of quantum $\mathfrak{gl}_n$ \cite{KazhdanWenzl}. Similar recognition theorems for other quantum group categories appear in \cite{TubaWenzl, Kuperberg, MPS1, MPS2, Copeland, CopelandEdieMichell, MW}. Note that our recognition theorem for the Delannoy category assumes that the category $\cC$ is symmetric; this is a stronger assumption than many of the above results, which only assume braided or make no commutativity assumption. It is interesting to ask if the symmetric assumption in Theorem \ref{thm:B} could be relaxed to braided. In this direction, we will show that Deligne's category $\uRep(\fS_t)$ satisfies such a braided recognition theorem in \cite{SnyderSpencer}.
\end{remark}

\subsection{General categories with trivial Adams operations}

As mentioned above, a natural problem is to investigate general pre-Tannakian categories with trivial Adams operations. In this direction, we prove the following result:

\begin{maintheorem} \label{thm:C}
Let $\cC$ be a semi-simple pre-Tannakian category and let $X$ be a generating object of $\cC$ that satisfies the following conditions:
\begin{enumerate}[(i)]
\item The class $[X]$ in $\rK(\cC)$ is fixed by $\psi^2$.
\item $X$ and $\lw^2{X}$ are simple.
\item $X$ is not a summand of the Schur functor $\bS_{(2,1)}(X)$.
\end{enumerate}
Then:
\begin{enumerate}
\item There is a (essentially unique) tensor functor $\Phi \colon \cD \to \cC$ satisfying $\Phi(L_{\bb})=X$.
\item The object $X$ has dimension $-1$, super exponential growth, and is fixed by $\psi^p$ for all $p$ different from the characteristic.
\item There is an oligomorphic group $(G, \Omega)$ with a regular measure $\mu$ such that $\cC$ is equivalent to the tensor category $\uRep(G, \mu)$.
\item The group $G$ preserves a total order on $\Omega$ and acts transitively on the set $\Omega^{(n)}$ of $n$-element subsets of $\Omega$ for $n=1,2,3$ (i.e., it is 3-homogeneous).
\end{enumerate}
\end{maintheorem}

In statement (c), we are appealing to the theory of \cite{repst} that associates a tensor category to an oligomorphic group equipped with a measure; this is briefly reviewed in \S \ref{ss:oligo}.

It would be nice if condition (iii) in the theorem could be relaxed. We make a few comments about this.
\begin{itemize}
\item We give some other conditions that imply (iii) in \S \ref{ss:further}. For instance, if (i) and (ii) hold and $[X]$ is also fixed by $\psi^3$ then (iii) holds (Proposition~\ref{prop:psi3-case1}). Similarly, if (i) and (ii) hold and $\lw^3{X}$ is a non-trivial simple then (iii) holds (Proposition~\ref{prop:wedge3-case1}).
\item There is an important example. Let $\HH$ be the non-abelian group of order~21 and let $V$ be the induction of a non-trivial one-dimensional representation of the 7-Sylow. Then $V$ satisfies (i) and (ii), but not (iii). Thus (iii) cannot simply be omitted in the theorem. See \S \ref{ss:case3} for details. We will see that this example is really quite fundamental in some ways. For instance, we show (Theorem~\ref{thm:mod}) that it is essentially the unique example of moderate growth that satisfies (i) and (ii).
\item Suppose (i) and (ii) hold, but (iii) does not; also assume $\lw^4{X} \ne 0$, to exclude the above example. This is what we call ``Case~IIb'' in our analysis. We do not know if this case can actually occur. If it does, it would give an interesting new pre-Tannakian category. We explore this possibility a little in \S \ref{s:caseIIb}, but do not come to any definitive conclusion; this is an interesting direction for future work.
\end{itemize}
Of course, it would also be nice to relax the condition that $\lw^2{X}$ is simple. This seems to be a much more difficult problem that would require serious new ideas.

The theorem suggests some interesting questions:

\begin{question}
If $\cC$ is a semi-simple pre-Tannakian category such that all Adams operations are trivial on $\rK(\cC)$, does $\cC$ come from an oligomorphic group $(G, \Omega)$? If so, can the triviality of the Adams operations be translated into group-theoretic properties of $G$? E.g., must $G$ preserve a total order on $\Omega$?
\end{question}

\begin{question}
Suppose $G$ is an oligomorphic group that preserves a total order, and we have a regular measure $\mu$ giving rise to a semi-simple tensor category $\cC=\uRep(G, \mu)$. Are the Adams operations on $\rK(\cC)$ trivial? The two most immediate cases to investigate are the category associated to $\GG \wr \GG$  mentioned in \S \ref{ss:adams}, and the Delannoy tree category constructed by Kriz \cite{kriz2}.
\end{question}

\subsection{Outline of proofs}

Let $X$ be as in Theorem~\ref{thm:C}. The hypotheses on $X$ are ``character theoretic'' in the sense that they are at the level of the Grothendieck group. The key step in the proof is to promote this information to structure: we show that $X$ carries the structure of an o-Frobenius algebra (\S \ref{ss:key}). One version of the universal property property for $\cD$ now shows that there is a functor $\Phi \colon \cD \to \cC$ satisfying $\Phi(L_{\bb})=X$. This proves statement (a) from Theorem~\ref{thm:C}, and (b) is a simple corollary of it. Statements (c) and (d) then follow from \cite{discrete}.

Now let $X$ be as in Theorem~\ref{thm:A}. Since $\lw^3{X}$ is a non-trivial simple, it follows that $X$ is not a summand of $\bS_{(2,1)}(X)$ (Proposition~\ref{prop:wedge3-case1}). Thus Theorem~\ref{thm:C} applies, and gives us a functor $\Phi \colon \cD \to \cC$ as above. It is now a rather simple matter to conclude that $\Phi$ is fully faithful from the assumption that all exterior powers of $X$ are simple; since $X$ generates $\cC$, it follows that $\Phi$ is an equivalence.

We now discuss the proof of Theorem~\ref{thm:B}. Suppose we have a semi-ring isomorphism $i \colon \rK_+(\cD) \to \rK_+(\cC)$. Let $[X]$ be the class in $\rK_+(\cC)$ corresponding to $[L_{\bb}]$. We do not know a priori that $i$ is compatible with the $\lambda$-ring structures on the Grothendieck groups; however, using the very strong constraints present in this situation, we deduce that the class of $\lw^2{L_{\bb}}$ does indeed map to that of $\lw^2{X}$. From this, we are able to verify that the hypotheses of Theorem~\ref{thm:C} are fulfilled, and so we get a functor $\Phi$ as above. With $\Phi$ in hand, it is not difficult to show that $i$ is compatible with exterior powers, at least on $L_{\bb}$. This shows that all exterior powers of $X$ are simple, from which we conclude that $\Phi$ is an equivalence (the argument is similar to the proof of Theorem~\ref{thm:A} at this point).

\begin{remark}
The basic structure of the above arguments can be summarized as follows. Given an object $X$ in $\cC$ together with some assumptions on maps between small tensor powers of $X$, there is an induced functor $\Phi \colon \cD \to \cC$; if, moreover, $\rK(\cC) \cong \rK(\cD)$ then this functor is an equivalence. This structure of argument is typical for recognition theorems, and employed in those mentoned in Remark~\ref{rem:recognition}. One notable difference in the present case is that we actually make assumptions on maps between small Schur functors of $X$, not just small tensor powers; this is something that only makes sense to do for symmetric tensor categories.
\end{remark}

\subsection{Tensor category terminology} \label{ss:tencat}

We fix an algebraically closed field of characteristic $\ne 2,3,5,7$ for the duration of the paper. A \defn{tensor category} is an additive $k$-linear category equipped with a $k$-bilinear symmetric monoidal structure. We write $\bone$ for the unit object in a tensor category, and $\Gamma(X)=\Hom(\bone, -)$ for the invariants functor. A tensor category is \defn{pre-Tannakian} if (a) it is abelian and all objects have finite length; (b) all $\Hom$ spaces are finite dimensional over $k$; (c) it is rigid, i.e., all objects have duals; and (d) $\End(\bone)=k$. An object $X$ \defn{generates} a pre-Tannakian category if every object is a subquotient of a direct sum of tensor products of $X$ and its dual.

\subsection{Notation}

We list some of the important notation here:
\begin{description}[align=right,labelwidth=2.5cm,leftmargin=!]
\item[ $k$ ] the coefficient field (algebraically closed of characteristic $\ne 2,3,5,7$)
\item[ $\bone$ ] the unit object of a tensor category
\item[ $\Gamma$ ] the functor $\Hom(\bone, -)$ on a tensor category
\item[ $\psi^p$ ] the $p$th Adams operator
\item[ $\cC$ ] a semi-simple pre-Tannakian category
\item[ $\cD$ ] the Delannoy category
\item[ $\GG$ ] the oligomorphic group $\Aut(\bR, <)$ used to define $\cD$
\item[ $L_{\lambda}$ ] the simple object of $\cD$ corresponding to $\lambda$
\item[ $\HH$ ] the non-abelian group of order~21
\end{description}

\subsection*{Acknowledgements}

We thank Nate Harman, Mikhail Khovanov, and Dylan Thurston for helpful discussions.

\section{Adams fixed objects} \label{s:initial}

In this section, we study objects fixed by the second Adams operation. The reasoning employed in this section is fairly formal, in the sense that it mostly occurs at the level of the Grothendieck group; in \S \ref{ss:key}, we will employ finer reasoning and obtain sharper results. We fix a semi-simple pre-Tannakian category $\cC$ for the duration of \S \ref{s:initial}.

\subsection{First comments}

We will say that an object $X$ of $\cC$ is fixed by $\psi^2$ if its class $[X]$ in the Grothendieck group is. Since $\cC$ is semi-simple, $X$ is fixed by $\psi^2$ if and only if there is an isomorphism
\begin{displaymath}
\Sym^2{X} \cong \lw^2{X} \oplus X.
\end{displaymath}
This is equivalent to asking for an isomorphism
\begin{displaymath}
X^{\otimes 2} \cong (\lw^2{X})^{\oplus 2} \oplus X
\end{displaymath}
Typically, there will not be canonical such isomorphisms. If $X$ is fixed by $\psi^2$ then, from the above isomorphism, $[\lw^2{X}]$ is a polynomial in $[X]$ in $\rK(\cC) \otimes \bQ$, and so $\lw^2{X}$ is also fixed by $\psi^2$. The above isomorphism has one other simple implication.

\begin{proposition} \label{prop:self-dual-fixed}
If $X$ is a self-dual simple fixed by $\psi^2$ then $X=\bone$.
\end{proposition}

\begin{proof}
$X^{\otimes 2}$ contains a unique copy of $\bone$, which can only occur in $X$.
\end{proof}

\subsection{Schur functors}

We now examine how Schur functors behave on Adams fixed objects. We begin in degree three. In what follows, we write $=$ for isomorphisms.

\begin{proposition} \label{prop:sym3}
If $X$ is fixed by $\psi^2$ then
\begin{displaymath}
\Sym^3{X} = \lw^3{X} \oplus (\lw^2{X})^{\oplus 2} \oplus X
\end{displaymath}
\end{proposition}

\begin{proof}
Since $X$ is fixed by $\psi^2$, we have
\begin{displaymath}
\Sym^2{X} = \lw^2{X} \oplus X.
\end{displaymath}
Now tensor with $X$ and apply the Pieri rule. We find
\begin{displaymath}
\Sym^3{X} \oplus \bS_{(2,1)}(X) = \lw^3{X} \oplus \bS_{(2,1)}(X) \oplus X^{\otimes 2}.
\end{displaymath}
The result now follows upon canceling the $\bS_{(2,1)}(X)$'s.
\end{proof}

Recall the following well-known plethysms \cite[\S 8, Example~6]{Macdonald}:
\begin{align*}
\lw^2{\lw^2{X}} &= \bS_{(2,1,1)}(X) &
\lw^2{\Sym^2{X}} &= \bS_{(3,1)}(X) \\
\Sym^2{\lw^2{X}} &= \lw^4{X} \oplus \bS_{(2,2)}(X) &
\Sym^2{\Sym^2{X}} &= \Sym^4{X} \oplus \bS_{(2,2)}(X).
\end{align*}
We now examine degree four Schur functors.

\begin{proposition} \label{prop:deg4}
Suppose $X$ is fixed by $\psi^2$. Then:
\begin{align*}
\lw^4{X} \oplus \bS_{(2,2)}(X) &= \bS_{(2,1,1)}(X) \oplus \lw^2{X} \\
\bS_{(3,1)}(X) &= \bS_{(2,1,1)}(X) \oplus \lw^3{X} \oplus \bS_{(2,1)}(X) \oplus \lw^2{X} \\
\Sym^4{X} &= \lw^4{X} \oplus \lw^3{X} \oplus \bS_{(2,1)}(X) \oplus \lw^2{X} \oplus X.
\end{align*}
\end{proposition}

\begin{proof}
Since $\lw^2{X}$ is fixed by $\psi^2$, we have
\begin{displaymath}
\Sym^2{\lw^2{X}} =\lw^2\lw^2{X} \oplus \lw^2{X}.
\end{displaymath}
Using the plethysm for the two compositions above yields the first formula. Next, we have
\begin{displaymath}
\lw^2{\Sym^2{X}} = \lw^2(\lw^2{X} \oplus X) = \lw^2(\lw^2{X}) \oplus (X \otimes \lw^2{X}) \oplus \lw^2{X}.
\end{displaymath}
Using the plethysms for the two compositions and the Pieri rule, yields the second formula. Finally, we have
\begin{displaymath}
\Sym^2(\Sym^2{X}) = \Sym^2(\lw^2{X} \oplus X) = \Sym^2(\lw^2{X}) \oplus (X \otimes \lw^2{X}) \oplus \Sym^2{X}.
\end{displaymath}
Using the plethsysms for the two compositions and the Pieri rule yields the third formula. (Note: we must cancel $\bS_{(2,2)}(X)$ from either side.)
\end{proof}

\begin{remark} \label{rmk:psi3}
If we only assume that $X$ is fixed by $\psi^2$ then there is not much we can say about $\bS_{(2,1)}(X)$. However, if $X$ is fixed by $\psi^2$ and $\psi^3$ then we have an isomorphism
\begin{displaymath}
\bS_{(2,1)}(X) = (\lw^3{X})^{\oplus 2} \oplus (\lw^2{X})^{\oplus 2}.
\end{displaymath}
Indeed, this follow from Proposition~\ref{prop:sym3} and the definition of $\psi^3$:
\begin{displaymath}
\psi^3([X]) = [\Sym^3{X}] - [\bS_{(2,1)}(X)] + [\lw^3{X}].
\end{displaymath}
In general, if $X$ is fixed by $\psi^{\ell}$ for all $\ell \le n$, and $n!$ is invertible in $k$, then all Schur functors of $X$ up to degree $n$ decompose into exterior powers. See \cite[\S 8.5]{line} for details.
\end{remark}

\subsection{Moderate growth}

We say that $X$ has \defn{moderate growth} if the length of $X^{\otimes n}$ grows at most exponentially with $n$; otherwise, $X$ has \defn{super-exponential growth}. If $\lw^n{X}=0$ for some $n$ then $X$ has moderate growth. See \cite[\S 4]{CEO} for this fact, and other results about moderate growth. We require the following result about Adams fixed objects of moderate growth.

\begin{proposition} \label{prop:moderate}
Suppose $X$ generates $\cC$, is fixed by $\psi^2$, and has moderate growth. Then $\cC$ is equivalent to $\Rep(\Gamma)$ where $\Gamma$ is a finite group scheme over $k$.
\end{proposition}

\begin{proof}
If $k$ has characteristic~0 then there is a fiber functor $\cC \to \sVec$ by a theorem of Deligne \cite{Deligne2}. If $k$ has positive characteristic $p$ there is a fiber functor $\cC \to \mathrm{Ver}_p$ by the main theorem of \cite{CEO}. Since the image of $X$ is fixed by $\psi^2$, it belongs to $\Vec$; this is easy to see directly in the characteristic~0 case, and is \cite[Corollary~4.4]{Czenky} in characteristic $p$. As $X$ generates $\cC$, the entire fiber functor takes values in $\Vec$. Thus $\cC$ is the representation category of an affine group scheme by Tannakian reconstruction.

Fix an equivalence $\cC=\Rep(\Gamma)$ where $\Gamma$ is an affine group scheme over $k$, and identify $X$ with a representation of $\Gamma$. Note that $X$ is a faithful representation of $\Gamma$ since $X$ generates $\cC$; in particular, $\Gamma$ is of finite type over $k$. Suppose we have a closed subgroup $\bG_m$ of $\Gamma$. The restriction of $X$ to $\bG_m$ decomposes into weights, which are indexed by $\bZ$. The operator $\psi^2$ acts on weights by multiplication by~2, and so the set of weights that appear is stable under multiplication by~2. Since there are only finitely many weights appearing in $X$, it follows that the only weight can be~0, meaning $\bG_m$ acts trivially. This contradicts faithfulness. We thus see that $\Gamma$ contains no $\bG_m$.

Now, $\Rep(\Gamma)=\cC$ is semi-simple. In characteristic~0, this implies that the identity component of $\Gamma$ is a reductive group \cite[Ch.~IV, Prop~3.3]{DemazureGabriel}. Since a non-trivial reductive group contains a copy of $\bG_m$, we see that the identity component of $\Gamma$ is trivial, and so $\Gamma$ is finite over $k$. In positive characteristic, semi-simplicity of $\Rep(\Gamma)$ implies that the identity component of $\Gamma$ is a diagonalizable group (Nagata's theorem) \cite[Ch.~IV, Prop.~3.6]{DemazureGabriel}. A diagonalizable group containing no $\bG_m$ is finite, and so once again $\Gamma$ is finite.
\end{proof}

Czenky \cite[Theorem~4.9]{Czenky} showed that if $\cC$ has finitely many simples and $\psi^2$ acts trivially on $\rK(\cC)$ then $\cC$ is the category of vector spaces. The above proposition shows that ``finitely many simples'' can be relaxed to ``moderate growth.''

\begin{corollary} \label{cor:moderate}
If $\cC$ has moderate growth (i.e., every object of $\cC$ has moderate growth) and $\psi^2$ acts trivially on $\rK(\cC)$ then $\cC$ is the category of vector spaces.
\end{corollary}

\begin{proof}
Let $X$ be an object of $\cC$, and let $\cC'$ be the tensor subcategory it generates. The proposition implies that $\cC'$ has finitely many simple objects, and is thus the category of vector spaces by Czenky's theorem. Thus $X$ is isomorphic to $\bone^{\oplus n}$ for some $n$. Since $X$ is arbitrary, the result follows.
\end{proof}

\subsection{A trichotomy} \label{ss:tri}

Let $X$ be an object of $\cC$ satisfying the following two conditions:
\begin{enumerate}[(i)]
\item $X$ is fixed by $\psi^2$.
\item $X$ and $\lw^2{X}$ are simple.
\end{enumerate}
We require some understanding of the implications of these assumptions to prove the main results of this paper. It is useful to separate three cases:
\begin{description}[align=right,labelwidth=2.25cm,leftmargin=!, font=\normalfont]
\item[Case I] $\lw^2{X}$ is not a summand of $X \otimes X^*$ and $\lw^4{X} \ne 0$.
\item[Case II] $\lw^2{X}$ is a summand of $X \otimes X^*$ and $\lw^4{X} \ne 0$.
\item[Case III] $\lw^4{X}=0$.
\end{description}
Clearly, we are in exactly one of the three cases. Case~III is exceptional, and studied in \S \ref{ss:case3}. We now focus on the first two cases. The following propositions constitute our initial analysis.

\begin{proposition} \label{prop:case1}
Suppose $X$ belongs to Case~I. Then:
\begin{enumerate}
\item No two of the five simple objects $\bone$, $X$, $X^*$, $\lw^2{X}$, and $\lw^2{X^*}$ are isomorphic.
\item We have a decomposition
\begin{displaymath}
X \otimes X^* = Y_1 \oplus Y_2 \oplus X \oplus X^* \oplus \bone,
\end{displaymath}
where $Y_1$ and $Y_2$ are simple, not isomorphic to each other, and not isomorphic to any of the five simples in (a). Moreover, $Y_1$ and $Y_2$ are either duals of each other or both self-dual.
\item $X$ is not a summand of $\lw^3{X}$ or $\bS_{(2,1)}(X)$.
\end{enumerate}
\end{proposition}

\begin{proposition} \label{prop:case2}
Suppose $X$ belongs to Case~II. Then:
\begin{enumerate}
\item No two of the five simple objects $\bone$, $X$, $X^*$, $\lw^2{X}$, and $\lw^2{X^*}$ are isomorphic.
\item We have a decomposition
\begin{displaymath}
X \otimes X^* = \lw^2{X} \oplus \lw^2{X^*} \oplus X \oplus X^* \oplus \bone.
\end{displaymath}
In particular, $X$ has dimension $-1$.
\item $X$ appears with multiplicity one in $\lw^3{X}$ or $\bS_{(2,1)}(X)$, and with multiplicity zero in the other.
\end{enumerate}
\end{proposition}

Based on the above proposition, it is useful to subdivide Case~II:
\begin{description}[align=right,labelwidth=2.5cm,leftmargin=!, font=\normalfont]
\item[Case IIa] $X$ is a summand of $\lw^3{X}$.
\item[Case IIb] $X$ is a summand of $\bS_{(2,1)}{X}$.
\end{description}
We will eventually see that Case IIa is impossible, and that in Case~I the objects $Y_1$ and $Y_2$ are dual to each other and $\dim{X}=-1$ (see \S \ref{ss:thmC}). We do not know if Case~IIb is possible; see \S \ref{s:caseIIb} for further discussion.

We now prove the propositions, which will take the remainder of \S \ref{ss:tri}. We assume $\lw^4{X}$ is non-zero in what follows, i.e., we are in Case~I or Case~II.

\begin{lemma}
No two of $\bone$, $X$, $X^*$, $\lw^2{X}$, and $\lw^2{X^*}$ are isomorphic.
\end{lemma}

\begin{proof}
It suffices to treat the following pairs:
\begin{center}
\begin{tabular}{lll}
(a) $\bone$ and $X$ & (b) $X$ and $X^*$ & (c) $\bone$ and $\lw^2{X}$ \\
(d) $\lw^2{X}$ and $\lw^2{X^*}$ & (e) $X$ and $\lw^2{X}$ & (f) $X^*$ and $\lw^2{X}$
\end{tabular}
\end{center}
We go through each in turn.

(a) If $X=\bone$ then $\lw^2{X}=0$, a contradiction.

(b) If $X$ were self-dual then $X=\bone$ (Proposition~\ref{prop:self-dual-fixed}), contradicting (a).

(c) If $\lw^2{X}=\bone$ then $X$ would be self-dual, contradicting (b).

(d) If $\lw^2{X}$ were self-dual then $\lw^2{X}=\bone$ (Proposition~\ref{prop:self-dual-fixed}), contradicting (c).

(e) Suppose $X = \lw^2{X}$. This implies $\dim{X}=3$. We have
\begin{displaymath}
X = \lw^2{X} = \lw^2 \lw^2{X} = \bS_{(2,1,1)}(X).
\end{displaymath}
Thus by Proposition~\ref{prop:deg4}, we find
\begin{displaymath}
\lw^4{X} \oplus \bS_{(2,2)}(X) = X \oplus \lw^2{X}.
\end{displaymath}
Now, both terms on the right are simple of dimension~3. Since $\lw^4{X}$ has dimension~0, it follows that $\lw^4{X}=0$, contradicting our running assumption.

(f) Suppose $X^* = \lw^2{X}$. This implies $\dim{X}=3$. We have
\begin{displaymath}
X = \lw^2{X^*} = \lw^2 \lw^2{X} = \bS_{(2,1,1)}(X).
\end{displaymath}
Then the rest of the argument proceeds identically to (e).
\end{proof}

\begin{lemma}
We have a decomposition
\begin{displaymath}
X \otimes X^* = Y_1 \oplus Y_2 \oplus X \oplus X^* \oplus \bone,
\end{displaymath}
where $Y_1$ and $Y_2$ are simple, and no two simples on the right are isomorphic. Moreover, $Y_1$ and $Y_2$ are either duals of each other or both self-dual.
\end{lemma}

\begin{proof}
From the decomposition $X^{\otimes 2} = X \oplus (\lw^2{X})^{\oplus 2}$ and the fact that $X$ and $\lw^2{X}$ are non-isomorphic, we see that $\End(X^{\otimes 2})$ is five dimensional. Thus $\End(X \otimes X^*)$ is also five dimensional, since it has the same dimension of $\End(X^{\otimes 2})$ by adjunction. Now, $X$ appears in $X^{\otimes 2}$, and so, by adjunction, $X$ appears in $X \otimes X^*$. By duality, $X^*$ also appears in $X \otimes X^*$. We thus have a decomposition
\begin{displaymath}
X \otimes X^* = Z \oplus X \oplus X^* \oplus \bone.
\end{displaymath}
for some object $Z$. Since $X \otimes X^*$ is self-dual and its endomorphism algebra is five dimensional, the result now follows.
\end{proof}

\begin{proof}[Proof of Proposition~\ref{prop:case1}]
We have already proved (a). Since $X$ belongs to Case~I, $\lw^2{X}$ is not a summand of $X \otimes X^*$. Since $X \otimes X^*$ is self-dual, it follows that $\lw^2{X^*}$ is not a summand either. Thus $Y_1$ and $Y_2$ are not isomorphic to any of the five simples in (a), which completes the proof of (b). By adjunction and the Pieri rule, we have
\begin{displaymath}
0 = \Hom(X \otimes X^*, \lw^2{X}) = \Hom(X, \lw^3{X} \oplus \bS_{(2,1)}(X)),
\end{displaymath}
and so $X$ does not occur in $\lw^3{X}$ or $\bS_{(2,1)}(X)$.
\end{proof}

\begin{proof}[Proof of Proposition~\ref{prop:case2}]
We have already proved (a). Since $X$ belongs to Case~II, $\lw^2{X}$ is a summand of $X \otimes X^*$. Relabeling if necessary, we can assume $Y_1=\lw^2{X}$. Since $\lw^2{X}$ is not self-dual, we must have $Y_2=Y_1^*$. This gives the decomposition occurring in (b), from which $\dim{X}=-1$ follows. Finally, as in the previous proof we have
\begin{displaymath}
\Hom(X \otimes X^*, \lw^2{X}) = \Hom(X, \lw^3{X} \oplus \bS_{(2,1)}(X)).
\end{displaymath}
Since the left side has dimension~1, part (c) follows.
\end{proof}

\subsection{Case~III} \label{ss:case3}

Let $\HH$ be the non-abelian group of order~21. This group is generated by elements $a$ and $b$ of orders~3 and~7 satisfying the relation $aba^{-1}=b^2$. The group $\HH$ has five conjugacy classes, represented by 1, $a$, $a^2$, $b$, and $b^3$. The character table of $\HH$ is given in Figure~\ref{fig:char-table}.

Let $V$ be one of the  three dimensional irreducible representations of $\HH$; the other one is the dual $V^*$ of $V$. The representation $\lw^3{V}$ is trivial: indeed, the character table shows that $V$ is the regular representation of the subgroup $\langle a \rangle$, and so $\lw^3{V}$ is the trivial representation of $\langle a \rangle$, and thus all of $\HH$. Thus $\lw^2{V} \cong V^*$ is irreducible. Since $b$ and $b^2$ are conjugate, we see directly from the character table that $V$ is fixed by $\psi^2$. Thus $V$ belongs to Case~III. There is an automorphism of $\HH$ that fixes $a$ and takes $b$ to $b^3$. This induces an auto-equivalence of the tensor category $\Rep(\HH)$ that switches $V$ and $V^*$.

\begin{figure}[!t]
\begin{displaymath}
\begin{array}{l|ccccc}
& 1 & a & a^2 & b & b^3 \\
\hline
\chi_1 & 1 & 1 & 1 & 1 & 1 \\
\chi_2 & 1 & \omega & \omega^2 & 1 & 1 \\
\chi_3 & 1 & \omega^2 & \omega & 1 & 1 \\
\chi_4 & 3 & 0 & 0 & z & z' \\
\chi_5 & 3 & 0 & 0 & z' & z \\
\end{array}
\end{displaymath}
\caption{The character table of $\HH$. Here $\omega$ is a primitive 3rd root of unity, and $z=\zeta+\zeta^2+\zeta^4 $ and $z'=\zeta^3+\zeta^5+\zeta^6$, where $\zeta$ is a primitive 7th root of unity.}
\label{fig:char-table}
\end{figure}

We now show that this is essentially the unique instance of Case~III.

\begin{proposition} \label{prop:case3}
Suppose $X$ generates $\cC$ and belongs to Case~III. Then there is an equivalence of tensor categories $\Phi \colon \Rep(\HH) \to \cC$ such that $\Phi(V)=X$.
\end{proposition}

\begin{proof}
Since $\lw^4{X}=0$, we see that $X$ has moderate growth. By Proposition~\ref{prop:moderate}, we thus have an equivalence $\cC=\Rep(\Gamma)$ where $\Gamma$ is a finite group scheme over $k$, and we identify $X$ with a representation of $\Gamma$ in what follows. Note that $X$ is a faithful representation of $\Gamma$ since $X$ generates $\cC$. Also, since $\lw^2{X}$ is simple but $\lw^4{X}$ vanishes, it follows that $X$ has dimension two or three.

Suppose $\mu_{\ell^r}$ is a closed subgroup of $\Gamma$, where $\ell$ is a prime\footnote{Note that if $\ell$ is different from the characteristic then $\mu_{\ell^r} \cong \bZ/\ell^r \bZ$, as $k$ is algebraically closed.}. The restriction of $X$ to $\mu_{\ell^r}$ decomposes into weights, which are indexed by $\bZ/\ell^r \bZ$; note that at least one weight has order $\ell^r$ since $X$ is faithful. The operator $\psi^2$ acts on weights by multiplication by~2, and so the set of weights that appear is stable under multiplication by~2. It follows that the order of~2 in $(\bZ/\ell^r \bZ)^{\times}$ is either~2 or~3, and so $\ell^r$ is either~3 or~7. In particular, if $k$ has positive characteristic $p$, there is no $\mu_p$ in $\Gamma$ (since we have excluded characteristics~3 and~7), and so the identity component of $\Gamma$ is trivial by Nagata's theorem \cite[Ch.~IV, Prop.~3.6]{DemazureGabriel}. It follows that $\Gamma$ is an ordinary finite group in which all elements have order dividing~21. Since the dimension of $X$ divides the order of $\Gamma$, we see that $X$ is three dimensional.

Let $H$ be a 3-Sylow subgroup of $\Gamma$. If $h \in H$ is non-trivial (and thus of order~3) then the above argument shows that the restriction of $X$ to $\langle h \rangle \cong \bZ/3\bZ$ is necessarily the regular representation. It now follows from character theory that $X$ is a direct sum of copies of the regular representation of $H$. We thus see that $H$ is either trivial or of order~3.

The third Sylow theorem now shows that $\Gamma$ has a unique (and thus normal) 7-Sylow subgroup $N$. Any irreducible representation of $N$ of dimension $<7$ is one dimensional. Thus, as a representation of $N$, we have a decomposition $X=L_1 \oplus L_2 \oplus L_3$, where the $L_i$'s are one dimensional. Put $L=L_1$. Since $X$ is fixed by $\psi^2$, we must have $L_2=L^{\otimes 2}$ and $L_3=L^{\otimes 4}$ (up to relabeling). Since $X$ is a faithful representation of $N$, it follows that $L$ is a faithful representation of $N$, and so $N$ is cyclic. Since every element of $N$ has order at most~7, it follows that $N$ has order at most~7.

Since $X$ is an irreducible three dimensional representation, $\Gamma$ cannot be abelian. We thus see that $\Gamma$ is isomorphic to $\HH$. Choosing an isomorphism, we obtain an equivalence $\Phi \colon \Rep(\HH) \to \cC$. Since $X$ is a three dimensional simple, we must have $\Phi(V)=X$ or $\Phi(V^*)=X$. In the latter case, we can pre-compose with an auto-equivalence of $\Rep(\HH)$ to put ourselves in the first case.
\end{proof}

\begin{corollary} \label{cor:case3}
Suppose $X$ belongs to Case~III. Then $\bS_{(2,1)}(X)$ contains $X$.
\end{corollary}

\begin{proof}
It suffices to prove this for $\cC=\Rep(\HH)$ and $X=V$. We have
\begin{displaymath}
\lw^2 V = V^* \qquad
V \otimes V^* = V \oplus V^* \oplus \bone \oplus \chi_2 \oplus \chi_3
\end{displaymath}
where $\chi_2$ and $\chi_3$ are the non-trivial one-dimensional representations. Thus
\begin{displaymath}
\bS_{(2,1)}(V) \oplus \lw^3 V = V \otimes \lw^2 V = V \oplus V^* \oplus \bone \oplus \chi_2 \oplus \chi_3.
\end{displaymath}
Since $V$ does not occur in $\lw^3{V}=\bone$, the result follows.
\end{proof}

\subsection{Futher analysis} \label{ss:further}

We now give some criteria to differentiate our three cases.

\begin{proposition} \label{prop:psi3-case1}
Let $X$ satisfy (i) and (ii) from \S \ref{ss:tri}. Suppose additionally that $X$ is fixed by $\psi^3$. Then $X$ belongs to Case~I.
\end{proposition}

\begin{proof}
If $X$ belonged to Case~III then the subcategory generated by $X$ would be equivalent to $\Rep(\HH)$, with $X$ corresponding to $V$ (Proposition~\ref{prop:case3}). But $V$ is not fixed by $\psi^3$ (look at the 3-Sylow). Thus we cannot be in this case.

As we explained in Remark~\ref{rmk:psi3}, we have an isomorphism
\begin{displaymath}
\bS_{(2,1)}(X) = (\lw^3{X})^{\oplus 2} \oplus (\lw^2{X})^{\oplus 2}.
\end{displaymath}
Since we are not in Case~III, $X$ is not isomorphic to $\lw^2{X}$. We thus see that $X$ appears in $\bS_{(2,1)}(X)$ if and only if it appears in $\lw^3{X}$. We therefore cannot be in Case~II, by Proposition~\ref{prop:case2}(c).
\end{proof}

\begin{proposition} \label{prop:caseIIb-wedge3}
If $X$ belongs to Case~IIb then $\lw^3{X}$ contains either $\lw^2{X}$ or $\lw^2{X^*}$ as a summand.
\end{proposition}

\begin{proof}
Suppose $X$ belongs to Case~IIb but $\lw^3{X}$ does not contain either $\lw^2{X}$ or $\lw^2{X^*}$; we will reach a contradiction. Let $m_{\lambda}$ and $n_{\lambda}$ be the multiplicities of $X$ and $\lw^2{X}$ in $\bS_{\lambda}(X)$. Note that $m_{(1,1,1)}=n_{(1,1,1)}=0$ and $m_{(2,1)}=1$ by our assumptions. 

To begin, we claim that
\begin{equation} \label{eq:mn}
n_{(2,1)} = 2+2m_{(2,1,1)}
\end{equation}
We have
\begin{displaymath}
\dim \Hom(X^{\otimes 2}, X \otimes \lw^2{X}) = m_{(1,1,1)}+m_{(2,1)}+2n_{(1,1,1)}+2n_{(2,1)} = 1+2n_{(2,1)}
\end{displaymath}
by using the irreducible decomposition of the source and the Pieri rule on the target. By adjunction, this is equal to
\begin{displaymath}
\dim \Hom(\lw^2{X}, X^{\otimes 2} \otimes X^*).
\end{displaymath}
The target here is
\begin{displaymath}
X \otimes (\bone \oplus X \oplus X^* \oplus \lw^2{X} \oplus \lw^2{X^*}).
\end{displaymath}
In the first three terms, $\lw^2{X}$ has multiplicities 0, 2, and~1 respectively. The multplicity of $\lw^2{X}$ in $X \otimes \lw^2{X}$ is $n_{(1,1,1)}+n_{(2,1)}=n_{(2,1)}$. Finally, the multiplicity of $\lw^2{X}$ in $X \otimes \lw^2{X^*}$ is the same as the multiplicity of $X$ in
\begin{displaymath}
(\lw^2{X})^{\otimes 2} = \lw^2{X} \oplus (\lw^2\lw^2{X})^{\oplus 2} = \lw^2{X} \oplus \bS_{(2,1,1)}(X)^{\oplus 2},
\end{displaymath}
which is $2m_{(2,1,1)}$. Thus, in total, we have
\begin{displaymath}
\dim \Hom(\lw^2{X}, X^{\otimes 2} \otimes X^*) = 3+n_{(2,1)}+2m_{(2,1,1)}.
\end{displaymath}
So, $1+2n_{(2,1)} =  3+n_{(2,1)}+2m_{(2,1,1)}$, which establishes \eqref{eq:mn}. In particular, $n_{(2,1)} \ge 2$.

We have
\begin{displaymath}
X \otimes X^* = \bone \oplus X^* \oplus X \oplus \lw^2{X} \oplus \lw^2{X^*}.
\end{displaymath}
The final three simples do not appear in $\lw^3{X}$ by assumption. Also, $\Hom(X^*, X^{\otimes 3}) = \Hom((X^*)^{\otimes 2}, X^{\otimes 2})$ vanishes since the source and target have no common simples. Similarly, $\Hom(\bone, X^{\otimes 3})=\Hom(X^*, X^{\otimes 2})$ vanishes. Thus neither $\bone$ nor $X^*$ appears in $\lw^3{X}$. We therefore find
\begin{displaymath}
0 = \Hom(X \otimes X^*, \lw^3{X}) = \Hom(X, \lw^4{X} \oplus \bS_{(2,1,1)}(X)),
\end{displaymath}
and so $m_{(1,1,1,1)}=m_{(2,1,1)}=0$. Proposition~\ref{prop:deg4} now implies $m_{(2,2)}=0$ and $m_{(3,1)}=1$. We thus see that $X$ has multiplicity one in $X \otimes \bS_{(2,1)}(X)$. Hence
\begin{displaymath}
\Hom(X, X \otimes \bS_{(2,1)}(X)) = \Hom(X \otimes X^*, \bS_{(2,1)}(X))
\end{displaymath}
is one dimensional. Since $X$ and $\lw^2{X}$ are summands of $X \otimes X^*$ and $m_{(2,1)}=1$, it follows that $n_{(2,1)}=0$. We now have a contradiction with \eqref{eq:mn}.
\end{proof}

\begin{proposition} \label{prop:wedge3-case1}
Let $X$ satisfy (i) and (ii) from \S \ref{ss:tri}. Suppose additionally that $\lw^3{X}$ is simple and not the tensor unit. Then $X$ belongs to Case~I.
\end{proposition}

\begin{proof}
We first claim that $\lw^3{X}$ is not isomorphic to $X$. Suppose, by way of contradiction, that this is the case. Let $\Sigma$ be the class of all objects in $\cC$ of the form
\begin{displaymath}
X^{\oplus a} \oplus (\lw^2{X})^{\oplus b} \oplus \bS_{(2,1)}(X)^{\oplus c}
\end{displaymath}
for $a,b,c \in \bN$. We have
\begin{equation} \label{eq:wedge3-case1}
(\lw^2{X})^{\oplus 2} \oplus X = X^{\otimes 2} = X \otimes \lw^3{X} = \lw^4{X} \oplus \bS_{(2,1,1)}(X)
\end{equation}
which shows that $\lw^4{X}$ and $\bS_{(2,1,1)}(X)$ belong to $\Sigma$ (with $c=0$). Proposition~\ref{prop:deg4} now shows that $\bS_{(2,2)}(X)$ belongs to $\Sigma$ (with $c=0$), and that $\bS_{(3,1)}(X)$ and $\Sym^4{X}$ also belong to $\Sigma$. We thus see that every degree four Schur functor of $X$ belongs to $\Sigma$.

Now, $X \otimes X$ is clearly in $\Sigma$, and $X \otimes \lw^2{X}$ is as well since $\lw^3{X}=X$. Since $X \otimes \bS_{(2,1)}(X)$ decomposes into a sum of degree~4 Schur functors, it too belongs to $\Sigma$. We thus see that if $M$ is any object of $\Sigma$ then $X \otimes M$ still belongs to $\Sigma$. It follows that $X^{\otimes n}$ belongs to $\Sigma$ for all $n$, which implies that $X$ has moderate growth. By Proposition~\ref{prop:moderate}, we can therefore regard $X$ as a finite dimensional representation of an affine group scheme.

The isomorphism $X \cong \lw^3{X}$ means that $X$ has (actual) dimension~4. Therefore $\lw^4{X}$ has dimension~1, and is thus simple. The identity \eqref{eq:wedge3-case1} thus implies that $\lw^4{X}$ is either $X$ or $\lw^2{X}$, which is impossible for dimension reasons. We therefore have a contradiction.

The above argument shows that $X$ does not belong to Case~IIa. If $X$ belonged to Case~IIb then by Proposition~\ref{prop:caseIIb-wedge3} $\lw^3{X}$ would be isomorphic to $\lw^2{X}$ or $\lw^2{X^*}$. However, this is a contradiction since $\dim{X}=-1$ (Proposition~\ref{prop:case2}(b)). Finally, in Case~III, $\lw^3{X}$ is the tensor unit (Proposition~\ref{prop:case3}), and so $X$ does not belong to this case.
\end{proof}

\section{The Delannoy category} \label{s:delannoy}

In this section, we review material from \cite{repst} on oligomorphic tensor categories in general, and material from \cite{line} on the Delannoy category specifically. We then discuss the universal property of the Delannoy category from the three different perspectives given in \cite{delmap}, \cite{kriz1}, and \cite{KS}. The material in \S \ref{ss:ofrob-to-etale} is new (though overlaps to some extent with \cite{KS}), while everything else in \S \ref{s:delannoy} is not.

\subsection{Oligomorphic tensor categories} \label{ss:oligo}

An \defn{oligomorphic group} is a permutation group $(G, \Omega)$ such that $G$ has finitely many orbits on $\Omega^n$ for all $n$. Fix such a group. We let $\bS(G)$ be the category of $G$-sets $X$ such that $G$ has finitely many orbits on $X$, and each orbit is a quotient of an orbit on some $\Omega^n$. This category is closed under finite products. See \cite[\S 2]{repst} for details.

A \defn{measure} on $G$ valued in $k$ is a rule that assigns to each morphism $f \colon Y \to X$ in $\bS(G)$, with $X$ transitive, a value $\mu(f)$ in $k$ such that some axioms hold (see \cite[\S 3.6]{repst}). Fix a measure. We then obtain a tensor category $\uPerm(G, \mu)$ as follows. The objects are formal symbols $\sC(X)$ where $X$ is an object of $\bS(G)$. A morphism $\sC(X) \to \sC(Y)$ is a $G$-invariant $k$-valued function on $Y \times X$. Composition is defined by convolution, the definition of which involves the measure $\mu$. We have
\begin{displaymath}
\sC(X) \oplus \sC(Y) = \sC(X \amalg Y), \qquad
\sC(X) \otimes \sC(Y) = \sC(X \times Y).
\end{displaymath}
Intuitively, $\sC(X)$ behaves like a permutation representation with basis indexed by $X$. The object $\sC(X)$ is self-dual of dimension $\mu(X)$, where $\mu(X)$ is shorthand for $\mu(X \to \mathrm{pt})$. See \cite[\S 8]{repst} for details.

The category $\uPerm(G, \mu)$ is essentially never abelian. If $\mu$ is \defn{regular} (meaning $\mu(f)$ is always non-zero) and nilpotent endomorphisms in $\uPerm(G, \mu)$ have trace zero, then the Karoubi envelope of $\uPerm(G, \mu)$ is a semi-simple pre-Tannakian category (see \cite[\S 13.2]{repst}), which we denote by $\uRep(G, \mu)$. When $\mu$ is not regular, finding pre-Tannakian categories related to $\uPerm(G, \mu)$ is a difficult and important problem, but it will not be relevant to this paper.

\subsection{The Delannoy category} \label{ss:delannoy}

Let $\GG$ be the group $\Aut(\bR, <)$ of order preserving self-bijections of the real line. This group acts oligomorphically on $\bR$. Let $\bR^{(n)}$ denote the subset of $\bR^n$ consisting of increasing tuples, i.e., points $(x_1, \ldots, x_n)$ with $x_1<x_2<\cdots<x_n$. The action of $\GG$ on $\bR^{(n)}$ is transitive, and these are all of the transitive $\GG$-sets in $\bS(\GG)$ \cite[Corollary~16.2]{repst}. There is a unique measure $\mu$ for $\GG$ satisfying $\mu(\bR^{(n)})=(-1)^n$. This measure is regular and nilpotent endomorphisms in $\uPerm(\GG, \mu)$ have trace~0 \cite[\S 16.5]{repst}. The \defn{Delannoy category} $\cD$ is $\uRep(\GG, \mu)$, i.e., the Karoubi envelope of $\uPerm(\GG, \mu)$. It is a semi-simple pre-Tannakian category, and its basic objects are the $\sC(\bR^{(n)})$.

A \defn{weight} is a word in the two-letter alphabet $\{\ww, \bb\}$. Given a weight $\lambda$ of length $n$, we define a simple summand $L_{\lambda}$ of $\sC(\bR^{(n)})$ in \cite[\S 4]{line}. We show that these simples account for all simple objects of $\cD$, and are mutually non-isomorphic. We explicitly determine the simple decomposition of $\sC(\bR^{(n)})$ in general \cite[Theorem~4.7]{line}. For example, we have
\begin{displaymath}
\sC(\bR) = L_{\bb} \oplus L_{\ww} \oplus \bone,
\end{displaymath}
where $\bone=L_{\emptyset}$ is the tensor unit. If $\lambda$ has length $n$ then $L_{\lambda}$ has dimension $(-1)^n$ \cite[Corollary~5.7]{line}. The dual of $L_{\lambda}$ is $L_{\lambda^{\vee}}$, where $\lambda^{\vee}$ is obtained from $\lambda$ by switching $\bb$ and $\ww$ \cite[Proposition~4.16]{line}. We have $\lw^n{L_{\bb}} = L_{\lambda}$ where $\lambda$ is the length $n$ word with all letters $\bb$ \cite[Proposition~8.5]{line}. As we have already said, every Adams operations on $\rK(\cD)$ is trivial, i.e., the identity map \cite[Theorem~8.2]{line}.

\begin{remark}
The group $\GG$ has four measures in total. The other three are not regular (or even quasi-regular). There has been some recent progress in understanding the categories associated to these measures \cite{delmap, fake}.
\end{remark}

\subsection{Ordered \'etale algebras} \label{ss:order-etale}

The Delannoy category has a universal property that we will require. In fact, there are a few variants. In \S \ref{ss:order-etale}, we discuss the one from \cite{delmap} that is based on \'etale algebras. Fix a semi-simple pre-Tannakian category $\cC$ in what follows.

An \defn{\'etale algebra} in $\cC$ is a commutative algebra object $A$ for which the trace pairing $A \otimes A \to \bone$ is perfect. See \cite[\S 4,5]{discrete} or \cite[\S 3]{delmap} for general background. Let $\Et(\cC)$ be the category of \'etale algebras in $\cC$, where morphisms are algebra homomorphisms. In \cite[Theorem~6.1]{discrete}, we show that $\Et(\cC)^{\op}$ is a \defn{pre-Galois category}, which means that it is equivalent to a category of the form $\bS(G)$ for some pro-oligomorphic group $G$; see \cite[\S 2]{repst} for the definition of pro-oligomorphic group, which is a mild generalization of oligomorphic group, and see \cite{pregalois} for more background on pre-Galois categories. Using this equivalence, any concept related to $G$-sets can be transferred to \'etale algebras.

An \defn{ordered \'etale algebra} is an \'etale algebra $A$ equipped with the extra structure required to give a $G$-invariant total order on the corresponding $G$-set. A bit more precisely, suppose $A$ corresponds to the $G$-set $X$. Then $A \otimes A$ corresponds to $X \times X$, and the decomposition of $A \otimes A$ into a product of simple \'etale algebras corresponds to the decomposition of $X \times X$ into $G$-orbits. Thus giving a binary relation on $X$ amounts to giving a direct factor of the algebra $A \otimes A$. The axioms for a total order can then be rephrased purely in terms of this factor algebra. See \cite[\S 7.2]{delmap} for details. The most important example of an ordered \'etale algebra is the object $\sC(\bR)$ of the Delannoy category. This corresponds to the $\GG$-set $\bR$, which carries its usual total order.

Let $A$ be an ordered \'etale algebra. We say that $A$ is \defn{1-Delannic}\footnote{In \cite{delmap}, the term ``Delannic of type~1'' is used.} if it satisfies three numerical conditions, which we now describe. The first condition is simply that $A$ has dimension $-1$. The other two conditions are slightly more complicated. Let $A^{(2)}$ be the factor algebra of $A \otimes A$ giving the total order. We have two maps $A \to A^{(2)}$, coming from the two maps $A \to A \otimes A$. We can thus regard $A^{(2)}$ as an $A$-module in two ways. The two other conditions are that $A^{(2)}$ has dimension $-1$ in the category of $A$-modules, for both module structures. The algebra $\sC(\bR)$ is 1-Delannic.

Tensor functors map \'etale algebras to \'etale algebras, and preserve all of the extra structure discussed above. We thus have a well-defined map
\begin{equation} \label{eq:delannic}
\{ \text{tensor functors $\cD \to \cC$} \}/{\sim} \to \{ \text{1-Delannic algebras in $\cC$} \}/{\sim}
\end{equation}
given by $\Phi \mapsto \Phi(\sC(\bR))$. The universal property for $\cD$ is the following:

\begin{theorem} \label{thm:delannic}
The map \eqref{eq:delannic} is a bijection. Thus giving a tensor functor $\cD \to \cC$ is equivalent to giving a 1-Delannic algebra in $\cC$.
\end{theorem}

See \cite[Theorem~7.13]{delmap} for details. We note that the results of \cite{delmap} are more general, in two ways. First, the target category is not required to be pre-Tannakian. And second, analogous statements are given for the categories associated to the other measures for $\GG$.

If $X$ is a totally ordered set, we can reverse the order to obtain a new totally ordered set. Similarly, if $A$ is an ordered \'etale algebra we can reverse the order to obtain a new order on $A$. This preserves 1-Delannic algebras \cite[\S 7.5]{delmap}. In particular, reversing the order on $\sC(\bR)$ gives a new 1-Delannic algebra in $\cD$, which, by the above theorem, corresponds to a tensor functor $\Pi \colon \cD \to \cD$. It is not difficult to see that $\Pi$ is an equivalence that squares to the identity. Moreover, one can show that $\Pi$ corresponds to the auto-equivalence discussed in \cite[Remark~5.17]{line}.

\subsection{o-Frobenius algebras} \label{ss:ofrob}

We now discuss a different universal property for $\cD$. This one comes from \cite{KS}, and is phrased in terms of the following kinds of objects:

\begin{definition} \label{defn:o-frob}
An \defn{o-Frobenius algebra} is an object $X$ of $\cC$ equipped with maps, called multiplication and comultiplication,
\begin{displaymath}
\mu \colon X \otimes X \to X, \qquad \delta \colon X \to X \otimes X
\end{displaymath}
such that the following conditions hold:
\begin{enumerate}
\item Multiplication is commutative and comultiplication is cocommutative.
\item Multiplication is associative and comultiplication is coassociative.
\item We have $\mu \delta = \id_X$.
\item We have the equation
\begin{displaymath}
\gamma+\gamma' = \id_{X \otimes X} + \delta \mu
\end{displaymath}
where $\gamma$ and $\gamma'$ are the endomorphisms of $X \otimes X$ defined by
\begin{displaymath}
\gamma = (\id \otimes \mu)(\delta \otimes \id), \qquad
\gamma' = (\mu \otimes \id)(\id \otimes \delta).
\end{displaymath}
\end{enumerate}
\end{definition}

\begin{remark} \label{rem:rescaling}
If $(\mu, \delta)$ is an o-Frobenius structure on $X$ and $\lambda$ is a non-zero scalar then $(\lambda \mu, \lambda^{-1} \delta)$ is also an o-Frobenius structure on $X$.
\end{remark}

Ordinary (special commutative) Frobenius algebras essentially encode an equality relation in the language of symmetric monoidal categories. o-Frobenius algebras instead encode an order. The following is the simplest example of this.

\begin{example} \label{ex:o-frob}
Let $S$ be a totally ordered set and let $X=k[S]$ be the vector space having for a basis symbols $e_a$ with $a \in S$. Let $e_{a,b}=e_a \otimes e_b$ denote the basis vectors of $X \otimes X$. Define
\begin{displaymath}
\mu(e_{a, b}) = e_{\min(a,b)}, \qquad \delta(e_a) = e_{a,a}.
\end{displaymath}
Then $X$ is an o-Frobenius algebra. The first three axioms are clear. We have
\begin{displaymath}
\gamma(e_{a, b}) = e_{a, \min(a,b)}, \qquad
\gamma'(e_{a, b}) = e_{\min(a,b),b}.
\end{displaymath}
The fourth axiom thus takes the form
\begin{displaymath}
e_{a, \min(a,b)} + e_{\min(a,b), b} = e_{a, b} + e_{\min(a,b), \min(a,b)},
\end{displaymath}
which is clearly true.
\end{example}

Fix an o-Frobenius algebra $X$. We now make a few simple observations.

\begin{proposition} \label{prop:gamma-idemp}
$\gamma$ and $\gamma'$ are commuting idempotents satisfying $\gamma \gamma' = \delta \mu$.
\end{proposition}

\begin{proof}
It follows immediately from the axioms that $\gamma$ and $\gamma'$ are idempotent. It also follows directly that $\mu \gamma=\mu$ and $\mu \gamma'=\mu$. Multiplying both sides in axiom (d) by $\gamma$ (on the right) now gives $\gamma \gamma' = \delta \mu$. Similarly we find $\gamma' \gamma = \delta \mu$. Thus $\gamma$ and $\gamma'$ commute.
\end{proof}

\begin{proposition} \label{prop:o-frob-psi2}
We have a canonical isomorphism
\begin{displaymath}
\Sym^2{X} = \lw^2{X} \oplus X.
\end{displaymath}
In particular, $X$ is fixed by $\psi^2$.
\end{proposition}

\begin{proof}
Let $e=\tfrac{1}{2} (1+\tau)$ and $f=\tfrac{1}{2} (1-\tau)$. We identify $\Sym^2{X}$ and $\lw^2{X}$ as summands of $X^{\otimes 2}$ via $e$ and $f$, and we identify $X$ with a summand of $X^{\otimes 2}$ via $\delta \mu$. We define maps
\begin{displaymath}
\phi \colon \Sym^2{X} \to \lw^2{X} \oplus X, \qquad x \mapsto (4f\gamma ex, \delta \mu x)
\end{displaymath}
\begin{displaymath}
\psi \colon \lw^2{X} \oplus X \to \Sym^2{X}, \qquad (y,z) \mapsto e \gamma f y + z.
\end{displaymath}
A straightforward calculation shows that $\phi$ and $\psi$ are inverse to one another.
\end{proof}

Define $\alpha \colon \bone \to X$ to be the composition
\begin{displaymath}
\xymatrix@C=3em{
\bone \ar[r]^-{\cv} & X^* \otimes X \ar[r]^-{1 \otimes \delta} & X^* \otimes X \otimes X \ar[r]^-{\ev \otimes \id} & X. }
\end{displaymath}
Analogously define a map $\beta \colon X \to \bone$. We say that $X$ is \defn{strict} if $\alpha$ and $\beta$ vanish. Note that if $X$ is a non-trivial simple then $X$ is automatically strict.

\begin{proposition} \label{prop:trace-gamma}
We have $\tr(\gamma)=\beta \alpha$ and $\tr(\gamma \tau) = \tr(\tau \gamma) = \dim{X}$. If $X$ is strict then $\dim{X}$ is either~0 or $-1$.
\end{proposition}

\begin{proof}
Recall that the \defn{partial trace} of a map $f \colon X \otimes Y \to X \otimes Y$, denoted $\ptr(f)$, is the composition
\begin{displaymath}
\xymatrix@C=3em{
X \ar[r]^-{\id \otimes \cv_Y} & X \otimes Y \otimes Y^* \ar[r]^-{f \otimes \id} & X \otimes Y \otimes Y^* \ar[r]^-{\id \otimes \ev_Y} & X. }
\end{displaymath}
Note that $\tr(f) = \tr(\ptr(f))$. 
Now, we have $\ptr(\gamma) = (\id \otimes \epsilon) \delta$. Thus, $\tr(\gamma) = \tr \ptr(\gamma) = \beta\alpha$. In the diagram calculus for tensor categories this calculation becomes:
$$
\tr(\gamma) \; = 
\begin{tikzpicture}[baseline=0]
\draw [white, use as bounding box] (-1.5,-2.5) rectangle (1.5,2.5);
\node[circle,draw] (delta) at (-.5,-.5) {$\delta$};
\node[circle,draw] (mu) at (.5,.5) {$\mu$};
\node (topright) at (.5,1.25) {};
\node (topleft) at (-.5,1.25) {};
\node (bottomleft) at (-.5,-1.25) {};
\node (bottomright) at (.5,-1.25) {};
\draw[->,thick] (delta.45) -- (0,0);
\draw[thick] (0,0) -- (mu.225);
\draw[thick,->] (mu.90) -- (topright.center);
\draw[thick,->] (delta.90) to (topleft.center);
\draw[thick] (bottomleft.center) to (delta.270);
\draw[thick,<-] (mu.270) to (bottomright.center);
\draw[thick] (topright.center) to[out=75,in=285, looseness=4] (bottomright.center);
\draw[thick] (topleft.center) to[out=105,in=255, looseness=4] (bottomleft.center);
\end{tikzpicture}
= 
\begin{tikzpicture}[baseline=0]
\draw [white, use as bounding box] (-1.5,-1.5) rectangle (1.5,1.5);
\node[circle,draw] (delta) at (-.5,-.5) {$\delta$};
\node[circle,draw] (mu) at (.5,.5) {$\mu$};
\draw[->,thick] (delta.45) -- (0,0);
\draw[thick] (0,0) -- (mu.225);
\draw[->,thick] (mu.90) to[out=90, in=315, looseness = 6] (mu.315);
\draw[->,thick] (delta.135) to[out=135, in=270, looseness = 6] (delta.270);
\end{tikzpicture} = \; \beta \alpha.$$

We also have $\ptr(\tau \gamma) = \mu \tau \delta = \mu \delta = \id$, so $\tr(\tau \gamma) = \tr(\ptr(\tau\gamma)) = \tr(\id) = \dim{X}$. Or in the diagram calculus:
$$
\tr(\tau \gamma) \; = 
\begin{tikzpicture}[baseline=0]
\draw [white, use as bounding box] (-2,-2.5) rectangle (2,2.5);
\node[circle,draw] (delta) at (-.5,-.5) {$\delta$};
\node[circle,draw] (mu) at (.5,.5) {$\mu$};
\node (midright) at (.5,1.25) {};
\node (topright) at (.5,1.75) {};
\node (midleft) at (-.5,1.25) {};
\node (topleft) at (-.5,1.75) {};
\node (bottomleft) at (-.5,-1.25) {};
\node (bottomright) at (.5,-1.25) {};
\draw[thick] (midleft.center) to (topright.center);
\draw[thick] (midright.center) to (topleft.center);
\draw[->,thick] (delta.45) -- (0,0);
\draw[thick] (0,0) -- (mu.225);
\draw[thick,->] (mu.90) -- (midright.center);
\draw[thick,->] (delta.90) to (midleft.center);
\draw[thick] (bottomleft.center) to (delta.270);
\draw[thick,<-] (mu.270) to (bottomright.center);
\draw[thick] (topright.center) to[out=22.5,in=270, looseness=2] (bottomright.center);
\draw[thick] (topleft.center) to[out=157.5,in=270, looseness=2] (bottomleft.center);
\end{tikzpicture}
=
\begin{tikzpicture}[baseline=0]
\draw [white,use as bounding box] (-1,-2.5) rectangle (1,2.5);
\node[circle,draw] (delta) at (0,-1) {$\delta$};
\node[circle,draw] (mu) at (0,1) {$\mu$};
\draw[->,thick] (delta.135) to[out=135,in=315] (mu.315);
\draw[->,thick] (delta.45) to[out=45,in=225] (mu.225);
\draw[->,thick] (mu.80) to[out=80, in= 280, looseness=4] (delta.280);
\end{tikzpicture}
= \; \dim{X}.
$$

Suppose now that $X$ is strict. Then $\tr(\gamma)=0$, and so $\tr(\gamma')=0$ as well, since $\gamma'$ is conjugate to $\gamma$. It now follows from axiom (d) that $\id_{X \otimes X} + \delta \mu$ has vanishing trace. On the other hand, its trace is $d^2+d$, where $d=\dim{X}$. The result follows.
\end{proof}

\begin{remark} \label{rem:AssocNotNeeded}
Note that (co-)associativity, axiom (b), is not used in the above proof. The calculation of traces above only uses axioms (a) and (c), while the calculation of $\dim X$ also uses axiom (d).
\end{remark}

Recall that $\sC(\bR)$ decomposes as $L_{\bb} \oplus \bone \oplus L_{\ww}$. One can show that $L_{\bb}$ carries the structure of an o-Frobenius algebra, which is strict of dimension $-1$. In \cite{KS}, it is proved that this is the universal such algebra. That is, giving a tensor functor $\Phi \colon \cD \to \cC$ is equivalent to giving an o-Frobenius algebra in $\cC$ that is strict of dimension $-1$, via $\Phi \mapsto \Phi(L_{\bb})$.

\subsection{From o-Frobenius to ordered \'etale} \label{ss:ofrob-to-etale}

We have just seen that functors $\cD \to \cC$ correspond to both 1-Delannic algebras in $\cC$ and to strict o-Frobenius algebras in $\cC$ of dimension $-1$. Thus 1-Delannic algebras in $\cC$ correspond to such o-Frobenius algebras. We now explain how this works (in one direction), in the special case that is relevant to us. We include this discussion for two reasons. First, it is enlightening to understand how this correspondence works directly (i.e., without passing through $\cD$). And second, as a practical matter: as of the time of writing, the paper \cite{KS} has not yet appeared. The discussion in this section means that this paper only logically depends on the universal property from \cite{delmap}.

Let $X$ be an o-Frobenius algebra such that $X$ and $\lw^2{X}$ are simple, and $\lw^4{X} \ne 0$. A few observations. First, $X$ is not the trivial object and so $X$ is necessarily strict. Second, $\dim{X}=-1$ by Proposition~\ref{prop:trace-gamma}, since $\dim{X}=0$ is impossible (simples have non-zero dimensions). And third, $X$ is fixed by $\psi^2$ (Proposition~\ref{prop:o-frob-psi2}). We thus see that $X$ belongs to either Case~I or Case~II from \S \ref{ss:tri}. Define
\begin{displaymath}
A = X \oplus \bone \oplus X^*.
\end{displaymath}
The following is the main result we wish to explain. 

\begin{proposition} \label{prop:ofrob-to-etale}
The object $A$ is admits an \'etale algebra structure, which is unique up to isomorphism. As an \'etale algebra, it admits exactly two orders, which are reverses of each other. Under either order, it is a 1-Delannic algebra.
\end{proposition}

The proof will take the remainder of \S \ref{ss:ofrob-to-etale}. Write $\mu$ and $\delta$ for the given multiplication and co-multiplication on $X$. By adjunction, we can convert $\mu$ and $\delta$ into maps
\begin{displaymath}
\mu^* \colon X \otimes X^* \to X^*, \qquad
\delta^* \colon X \otimes X^* \to X.
\end{displaymath}
We now define an algebra structure on $A$. The summand $\bone$ is defined to be the unit, and $X$ and $X^*$ are defined to be (non-unital) subalgebras; here $X^*$ is an algebra by dualizing the co-multiplication on $X$. The multiplication map on $X \otimes X^* \subset A \otimes A$ is defined to be $\delta^* - \ev + \mu^*$, where $\ev$ is the evaluation map. The multiplication on $X^* \otimes X$ is defined similarly, so that the product on $A$ is commutative. These definitions are chosen to agree with what happens in the Delannoy category.

\begin{lemma} \label{lem:ofrob-to-etale-1}
The product on $A$ is associative.
\end{lemma}

\begin{proof}
We must show that the two multiplication maps $UVW \to A$ agree, where $U$, $V$, and $W$ can each be one of the three summands of $A$, and $UVW$ denotes their tensor product. This is clear if any one of $U$, $V$, or $W$ is $\bone$. It is also clear if $U$, $V$, and $W$ are all $X$ or all $X^*$, since the product on $X$ and $X^*$ is associative.

We now handle the case $U=V=X$ and $W=X^*$. The maps $UVW \to A$ have three components, corresponding to the three summands in the target. The $\bone$ components correspond, by adjunction, to two maps $XX \to X$, and one easily sees that both are equal to $-\mu$. The $X^*$ components correspond, by adjunction, to two maps $XXX \to X$, which one easily sees are the two maps built from $\mu$; they agree since $\mu$ is associative. Finally, the $X$ components correspond, by adjunction, to two maps $XX \to XX$. Call these two maps $f_1$ and $f_2$, where $f_1$ comes from multiplying $U$ and $V$ first, and $f_2$ from multiplying $V$ and $W$ first. One finds that $f_1=\delta \mu$ and $f_2 = \gamma+\gamma'-\id$. We have $f_1=f_2$ by Definition~\ref{defn:o-frob}(d), and so this case is complete.

The case $U=W=X$ and $V=X^*$ now formally follows using the fact that multiplication is commutative. Indeed, we have
\begin{displaymath}
(uv)w = w(uv) = (wu)v = (uw)v = u (wv) = u (vw).
\end{displaymath}
Here each expression denotes a map $UVW \to A$ built from multiplication and the symmetry, and we have only used commutativity of multiplication and the special case of associativity established in the previous paragraph.

All other remaining cases of associativity are similar to the above two.
\end{proof}

\begin{lemma} \label{lem:ofrob-to-etale-2}
The algebra $A$ is \'etale.
\end{lemma}

\begin{proof}
Let $\epsilon \colon A \to \bone$ be the trace map on $A$. Since $X$ is a non-trivial simple object, $\epsilon$ restricts to~0 on $X$ and $X^*$. We have $\epsilon(1)=\dim{A}=-1$, and so $\epsilon$ is non-zero on $\bone$. Now, it is clear that $\bone$ is orthogonal to $X$ and $X^*$ under the trace pairing on $A$ (since any map $X \to \bone$ is zero), and that this pairing is perfect on $\bone$. Since $X^{\otimes 2}$ only contains the simples $X$ and $\lw^2{X}$, which are non-trivial (\S \ref{ss:tri}), there is no non-zero map $X^{\otimes 2} \to \bone$, and so $X$ is self-orthogonal under the trace pairing on $A$. Similarly $X^*$ is self-orthogonal. The formulas for the multiplication on $A$ and for $\epsilon$ shows that the trace pairing between $X$ and $X^*$ is the canonical pairing. We thus see that the trace pairing on $A$ is perfect, and so $A$ is \'etale.
\end{proof}

Note that since the $\bone$ component of the trace is $-1$, the restriction of $\epsilon \circ \mu$ to $X \otimes X^*$ is $(-1) \cdot (-1) \cdot \ev = \ev$, which explains the appearance of $- \ev$ in the definition of multiplication above.

\begin{lemma} \label{lem:ofrob-to-etale-3}
The dimension of $\Gamma(A^{\otimes n})$ is 1, 3, and 13 for $n=1,2,3$.
\end{lemma}

\begin{proof}
We have
\begin{displaymath}
\dim \Gamma(X^{\otimes 2}) = \dim \Gamma(X^{\otimes 3}) = 0, \qquad \dim \Gamma(X^* \otimes X) = \dim \Gamma(X^* \otimes X^{\otimes 2}) = 1,
\end{displaymath}
as we now explain. We know the simple decompositions of $X^{\otimes 2}$ and $X^* \otimes X$ from \S \ref{ss:tri}, which explains those cases. Next, we have
\begin{displaymath}
\Hom(\bone, X^{\otimes 3}) = \Hom(X^*, X^{\otimes 2}) = 0
\end{displaymath}
and
\begin{displaymath}
\Hom(\bone, X^* \otimes X^{\otimes 2}) = \Hom(X, X^{\otimes 2}) = k,
\end{displaymath}
where, again, we are simply making use of the simple decomposition of $X^{\otimes 2}$. This explains the other two cases. The result now follows by simply expanding the tensor powers of $A^{\otimes n}$ and using the above formulas (and similar ones obtained from duality).
\end{proof}

For a set $U$, we let $U^{[n]}$ be the subset of $U^n$ where the coordinates are pairwise distinct. The symmetric group $\fS_n$ acts on $U^n$ by permuting coordinates, and the subset $U^{[n]}$ is stable under this action.

\begin{lemma} \label{lem:ofrob-to-etale-4}
Let $U$ be a $G$-set for some group $G$. Suppose that:
\begin{enumerate}[(i)]
\item The number of $G$-orbits on $U$, $U^{[2]}$, and $U^{[3]}$ is~1, 2, and 6.
\item The two $G$-orbits on $U^{[2]}$ are interchanged by $\fS_2$.
\item The cardinality of $U$ is at least four.
\end{enumerate}
Then either $G$-orbit on $U^{[2]}$ defines a total order on $U$.
\end{lemma}

\begin{proof}
Fix an orbit $R$ on $U^{[2]}$. For $x,y \in X$, define $x \le y$ if $(x,y) \in R$, or $x=y$. This relation is reflexive, anti-symmetric, and total. We must show that it is also transitive. Let $\Gamma$ be a tournament on the set $\{1,2,3\}$, i.e., a complete graph where each of the edges has been given a direction. Define $V_{\Gamma}$ to be the subset of $U^{[3]}$ consisting of tuples $(x_1, x_2, x_3)$ such that $x_i<x_j$ if there is an edge pointing from $i$ to $j$ in $\Gamma$. Clearly, $V_{\Gamma}$ is a $G$-stable subset and $U^{[3]}$ is the disjoint union of the $V_{\Gamma}$'s. Moreover, the action of $\fS_3$ on $U^{[3]}$ permutes the $V_{\Gamma}$'s in the natural manner.

Now, there are eight tournaments on three vertices, two of which are 3-cycles. Since $U^{[3]}$ has exactly six $G$-orbits, we see that there are two possibilities:
\begin{enumerate}
\item $V_{\Gamma}$ is empty if $\Gamma$ is a 3-cycle and a transitive $G$-set otherwise.
\item $V_{\Gamma}$ contains three $G$-orbits if $\Gamma$ is a 3-cycle and is empty otherwise.
\end{enumerate}
In case (a), the transitive law holds. We must therefore show that case (b) is impossible.

Suppose we are in case (b). Consider distinct elements $x,y,z \in U$ satisfying $x<y$ and $y<z$. We cannot have $x<z$, as this would imply that $(x,y,z)$ belongs to $V_{\Gamma}$ where $\Gamma$ is not a 3-cycle, but such $V_{\Gamma}$ are empty. We also cannot have $x=z$ since then we would have $x<y$ and $y<x$. We thus see that $z<x$. We therefore find that $<$ satisfies the following variant of transitivity: $x<y$ and $y<z$ implies $z<x$. However, it is not difficult to show that no such relation exists on a set with at least four elements, and this completes the proof.
\end{proof}

\begin{lemma} \label{lem:ofrob-to-etale-5}
The \'etale algebra $A$ carries two orders, which are reverses of each other. Under each, it is 1-Delannic.
\end{lemma}

\begin{proof}
Let $G$ be a pro-oligomorphic group such that $\Et(\cC)^{\op}$ is equivalent to $\bS(G)$, and let $U$ be the $G$-set corresponding to $A$. The number of $G$-orbits on $U^n$ is equal to the dimension of $\Gamma(A^{\otimes n})$, which is 1, 3, and 13 for $n=1,2,3$ by Lemma~\ref{lem:ofrob-to-etale-3}. We have isomorphisms
\begin{displaymath}
U^2 \cong U^{[2]} \amalg U, \qquad U^3 \cong U^{[3]} \amalg (U^{[2]})^{\amalg 3} \amalg U,
\end{displaymath}
from which we see that the number of $G$-orbits on $U^{[n]}$ is 1, 2, and 6 for $n=1,2,3$. It is clear from the formula for $A$ that the $\fS_2$ action on $\Gamma(A^{\otimes 2})$ is non-trivial, and so it follows that $\fS_2$ permutes the two $G$-orbits on $U^{[2]}$. Lemma~\ref{lem:ofrob-to-etale-3} thus shows that $U$ carries exactly two $G$-invariant total orders, which are reverses of each other. Equip $A$ with either order. Since $\Gamma(A)=k$ and $\dim{A}=-1$, it follows that $A$ is 1-Delannic by \cite[Proposition~7.10]{delmap}.
\end{proof}

\begin{lemma} \label{lem:ofrob-to-etale-6}
Any \'etale algebra structure on $A$ is isomorphic to the one constructed above.
\end{lemma}

\begin{proof}
The unit map is necessarily a scalar multiple of the inclusion $\bone \to A$; composing with a linear automorphism of $A$, we can assume it is the standard inclusion. The multiplication map $X \otimes X \to A$ is a scalar multiple $a$ of $\mu$, since the space $\Hom(X \otimes X, A)$ is one dimensional. Similarly, the multiplication map $X^* \otimes X^* \to A$ is a scalar multiple $a^*$ of the dual of $\mu$. Now, consider the multiplication map $X \otimes X^* \to A$. It has the form
\begin{displaymath}
b \delta^* - c\ev + b^* \mu^*
\end{displaymath}
for scalars $b$, $b^*$, and $c$. The trace on $A$ is a scalar multiple of the projection onto the $\bone$ summand, and so (as in the proof of Lemma~\ref{lem:ofrob-to-etale-2}) we must have $c \ne 0$ for the trace pairing to be non-degenerate.

We now analyze the associativity condition for the multiplication on $XXX^*$, as in the second paragraph of the proof of Lemma~\ref{lem:ofrob-to-etale-1}. The $\bone$ components of the two maps correspond, by adjunction, to two maps $XX \to X$, which are $-ac \mu$ and $-cb \mu$. From this, we conclude $a=b$, and a symmetrical argument gives $a^*=b^*$. The $X$ components correspond to two maps $XX \to XX$; label them $f_1$ and $f_2$, as in the proof of Lemma~\ref{lem:ofrob-to-etale-1}. We find $f_1=a a^* \delta \mu$ and $f_2=a a^* \gamma' + a a^* \gamma - c \cdot \id$, and so
\begin{displaymath}
c \cdot \id + aa^* \delta \mu = (a^*)^2 \gamma + aa^* \gamma'.
\end{displaymath}
By looking at the images of the idempotents $\id$, $\gamma$, and $\delta \mu$, we see that they are linearly independent, and thus Definition~\ref{defn:o-frob}(d) is the only linear relation between $\id$, $\delta \mu$, $\gamma$, and $\gamma'$, up to scaling. Since $c \ne 0$, we thus see that the above linear relation is simply the one in Definition~\ref{defn:o-frob}(d) scaled by $c$. We conclude $aa^*=c$. This shows $a \ne 0$ and $a^* \neq 0$.

Now $A$ has a $2$-parameter family of isomorphic algebra structures with fixed unit, by rescaling multiplication by $r$ on the $X$ and by $r^*$ on the $X^*$ factor. This shows that the algebra with scalars $(a,a^*,c)$ is isomorphic to the one with scalars $(ra, sa^*, rsc)$, so choosing $r = a^{-1}$, $s=(a^*)^{-1}$, we get that our algebra with scalars $(a, a^*, aa^*)$ is isomorphic to the one above with scalars $(1,1,1)$.
\end{proof}

\begin{remark}
The proof of Lemma~\ref{lem:ofrob-to-etale-6} shows that $A$ admits a unique \'etale algebra structure such that the standard inclusion of $\bone$ is the unit, the product on $X \otimes X$ is $\mu$, and the product on $X^* \otimes X^*$ is the dual of $\delta$. The third condition here could also be replaced with the condition: the trace pairing between $X$ and $X^*$ is the natural pairing.
\end{remark}

\begin{remark}
With more care, one can give a natural bijection between the two orders on $A$ and the two non-trivial simple constituents of $A$. In Delannoy, the simple constituents are the images of the indicator functions of $x \leq y$ and $x \geq y$ where $x$ is the input variable and $y$ the output variable, so we can assign the former simple to the ordering $<$ and the latter to the ordering $>$.
\end{remark}

\begin{remark}
Suppose $X$ is a strict o-Frobenius algebra of dimension $-1$, but allow $X$ and $\lw^2{X}$ to be non-simple. It is still true that $A=X \oplus \bone \oplus X^*$ naturally has the structure of a 1-Delannic algebra, though the proof is now somewhat more complicated. See \cite{KS} for details.
\end{remark}

\subsection{Kriz algebras}

We have seen two universal properties for the Delannoy category, one using ordered \'etale algebras and one using o-Frobenius algebras. In fact, there is yet another version, due to S.~Kriz \cite[\S 4.8]{kriz1}, which was historically first. For the sake of completeness, we briefly recall the statement. A \defn{Kriz algebra} in $\cC$ is an \'etale algebra $A$ of dimension $-1$ equipped with a decomposition
\begin{displaymath}
A = X \oplus \bone \oplus Y,
\end{displaymath}
where $X$ and $Y$ are (non-unital) subalgebras such that $X$ and $Y$ are isotropic and orthogonal to $\bone$ under the trace pairing\footnote{In characteristic~2, there is an additional condition; see \cite[\S 4.8]{kriz1}.}. These conditions imply that the trace pairing induces a perfect pairing between $X$ and $Y$. The algebra $\sC(\bR)$ in the Delannoy category is a Kriz algebra, with $X=L_{\bb}$ and $Y=L_{\ww}$. Kriz's theorem is that this Kriz algebra is universal, that is, tensor functors $\Phi \colon \cD \to \cC$ correspond to Kriz algebras in $\cC$, via $\Phi \mapsto \Phi(\sC(\bR))$. It follows from the universal property of $1$-Delannic algebras, that any $1$-Delannic algebra is Kriz, and the reverse follows from the universal property of Kriz algebras. It follows from both universal properties that an algebra is $1$-Delannic if and only if it is Kriz.

Note that the Kriz condition concerns subobjects of $A$, and thus concerns idepmpotents in $\Hom(A,A)$ under composition, while $1$-Delannic concerns idempotents in $\Gamma(A^{\otimes 2})$ under multiplication. Although the self-duality of $A$ identifies $\Hom(A,A)$ with $\Gamma(A^{\otimes 2})$ as vector spaces, this vector space isomorphism does not induce an isomorphism between these algebra structures. Instead in the Delannoy category, the minimal idempotents in $\Gamma(A^{\otimes 2})$ used to define its $1$-Delannic structure are $\delta_{x<y}$, $\delta_{y<x}$, and $1$, while the minimal idempotents in $\Hom(A,A)$ used to define its Kriz structure are $\delta_{x \leq y}$, $\delta_{y \leq x}$, and $\delta_{x=y} = \id.$ Because of this change of basis a direct proof that $A$ is $1$-Delannic if and only if it is Kriz is non-trivial.

\section{The main theorems} \label{s:main}

In this section we prove the main theorems of this paper. We fix a semi-simple pre-Tannakian category $\cC$ throughout \S \ref{s:main}.

\subsection{The key result} \label{ss:key}

We have seen that giving a functor $\cD \to \cC$ is equivalent to giving a certain kind of o-Frobenius algebra in $\cC$ (\S \ref{ss:ofrob}). On the other hand, Theorem~\ref{thm:C} states that if we have an object $X$ of $\cC$ satisfying certain character theoretic conditions then we get a functor $\cD \to \cC$. (By ``character theoretic'' we mean they take place at the level of the Grothendieck group.) The following theorem promotes the character theoretic conditions to the kind of structure required to apply the universal properties. It is the key result underlying the main theorems of this paper.

\begin{theorem} \label{thm:key}
Let $X$ be an object of $\cC$ such that the following conditions hold:
\begin{enumerate}[(i)]
\item $X$ is fixed by $\psi^2$.
\item $X$ and $\lw^2{X}$ are simple.
\item $X$ does not appear as a summand $\bS_{(2,1)}(X)$.
\end{enumerate}
Then $X$ admits an o-Frobenius structure, which is unique up to isomorphism.
\end{theorem}

The rest of \S \ref{ss:key} is devoted to the proof. We fix $X$ as in the statement of the theorem, which belongs to one of the three cases in \S \ref{ss:tri}. By Corollary~\ref{cor:case3}, $X$ does not belong to Case~III. Thus by Propositions~\ref{prop:case1} and~\ref{prop:case2}, we see that $\bone$, $X$, $X^*$, $\lw^2{X}$, and $\lw^2{X^*}$ are non-isomorphic simple objects. We have simple decompositions
\begin{displaymath}
\Sym^2(X) = \lw^2{X} \oplus X, \qquad X^{\otimes 2} = (\lw^2{X})^{\oplus 2} \oplus X.
\end{displaymath}
It follows that there are non-zero maps
\begin{displaymath}
\mu \colon X \otimes X \to X, \qquad \delta \colon X \to X \otimes X
\end{displaymath}
which are unique up to scaling. We fix $\mu$ and $\delta$ such that $\mu \delta = \id_X$. Note that $\delta \mu$ is an idempotent endomorphism of $X \otimes X$ with image $X$. Clearly, $\mu$ and $\delta$ are (co-)commutative. We must now verify the remaining axioms of Definition~\ref{defn:o-frob}.

\begin{lemma} \label{lem:key-11}
$\mu$ and $\delta$ are (co-)associative.
\end{lemma}

\begin{proof}
Let $f \colon X^{\otimes 3} \to X$ be the map $\mu \circ (\mu \otimes \id )$. We have a decomposition
\begin{displaymath}
X^{\otimes 3} = \Sym^3{X} \oplus \bS_{(2,1)}(X)^{\oplus 2} \oplus \lw^3{X}.
\end{displaymath}
Since $\mu$ is commutative, $f$ is invariant under the transposition $(1\;2)$, and so $f$ restricts to zero on $\lw^3{X}$. Since $X$ is not a constituent of $\bS_{(2,1)}(X)$, we see that $f$ also restricts to zero on the second summand above. It follows that $f$ factors through $\Sym^3(X)$, which means that it is invariant under the action of the symmetric group $\fS_3$. Since $\mu$ is $\fS_2$-invariant and $\mu \circ (\mu \otimes \id)$ is $\fS_3$-invariant, it follows that $\mu$ is associative. A similar argument shows that $\delta$ is co-associative.
\end{proof}

It remains to check Definition~\ref{defn:o-frob}(d). Let $\gamma$ and $\gamma'$ be as defined there, and let $\tau$ be the symmetry on $X^{\otimes 2}$. Note that $\gamma'$ is simply $\tau \gamma \tau$. Some of the general results we proved about o-Frobenius algebras can be easily deduced in our setting.

\begin{lemma} \label{lem:key-2}
We have the following:
\begin{enumerate}
\item The map $\gamma$ is idempotent.
\item We have $\gamma \cdot \delta \mu = \delta \mu$.
\item We have $\tr(\gamma)=0$ and $\tr(\gamma \tau)=\dim(X)$ for $i=1,2$.
\end{enumerate}
\end{lemma}

\begin{proof}
(a) and (b) are direct calculations, as in Proposition~\ref{prop:gamma-idemp}, and (c) follows as in Proposition~\ref{prop:trace-gamma}. Note that none of these results rely on the condition Definition~\ref{defn:o-frob}(d). Also note that $\alpha$ and $\beta$ (defined before Proposition~\ref{prop:trace-gamma}) vanish in the present setting since $X$ is a non-trivial simple.
\end{proof}

\begin{lemma} \label{lem:key-3}
We have $\dim{X}=-1$.
\end{lemma}

\begin{proof}
Let $d=\dim{X}$. The image of $\gamma$ is a subobject of $X \otimes X$ containing $X$ (by Lemma~\ref{lem:key-2}(b)), and thus of the form $X \oplus (\lw^2{X})^{\oplus r}$ where $r \in \{0,1,2\}$. Since $\gamma$ is idempotent, its trace is the dimension of its image. Since its trace is~0 (Lemma~\ref{lem:key-2}(c)), we find
\begin{displaymath}
0 = d+r \binom{d}{2}.
\end{displaymath}
If $r$ were~0 or~2, we would obtain $d=0$, which is not possible; a simple object in a semi-simple category has non-zero dimension. Thus $r=1$. The above equation then implies $d=-1$, since, as we have already remarked, $d=0$ is not possible.
\end{proof}

The following lemma completes the proof of the theorem.

\begin{lemma} \label{lem:key-4}
We have $\gamma+\gamma'=1+\delta m$.
\end{lemma}

\begin{proof}
Let $f=\tfrac{1}{2} (1-\tau)$, which is an idempotent of $R=\End(X^{\otimes 2})$. The ring $fRf$ is the endomorphism algebra of $\lw^2{X}$, and thus one dimensional and spanned by $f$. We therefore have an equation
\begin{equation} \label{eq:f}
f \gamma f = c f
\end{equation}
for some scalar $c$. We have
\begin{displaymath}
\tr(f \gamma f) = \tfrac{1}{4} \tr(\gamma - \tau \gamma - \gamma \tau + \tau \gamma \tau) = \tfrac{1}{2},
\end{displaymath}
where we have used Lemmas~\ref{lem:key-2} and~\ref{lem:key-3}, and the cyclic invariance of trace. We also have
\begin{displaymath}
\tr(f) = \dim(\lw^2{X}) = 1,
\end{displaymath}
by Lemma~\ref{lem:key-3}. Comparing the traces on the two sides of \eqref{eq:f}, we conclude $c=\tfrac{1}{2}$.

Let $e=\tfrac{1}{2} (1+\tau)$. The ring $eRe$ is the endomorphism algebra of $\Sym^2{X}=\lw^2{X} \oplus X$, and is thus two dimensional with basis $e$ and $\delta \mu$. We have $\gamma \cdot \delta \mu = \delta \mu$ (Lemma~\ref{lem:key-2}), and so $e \gamma e - \delta \mu$ is an element of $eRe$ that is killed by $\delta \mu$. We therefore have an equation
\begin{equation} \label{eq:e}
e \gamma e - \delta \mu = c' (e-\delta \mu)
\end{equation}
for some scalar $c'$. We have
\begin{displaymath}
\tr(e \gamma e) = \tfrac{1}{4} \tr(\gamma + \tau \gamma + \gamma \tau + \tau \gamma \tau) = - \tfrac{1}{2},
\end{displaymath}
as before. We also have
\begin{displaymath}
\tr(\delta \mu) = \dim{X}=-1, \qquad \tr(e)=\dim(\Sym^2{X})=0.
\end{displaymath}
Comparing traces of the two sides in \eqref{eq:e}, we find $c'=\tfrac{1}{2}$.

Adding \eqref{eq:f} and \eqref{eq:e} gives the stated identity.
\end{proof}

\subsection{Theorem~\ref{thm:C}} \label{ss:thmC}

We break the proof of the theorem into two parts. Here is the first:

\begin{theorem} \label{thm:C-1}
Suppose that $X$ is an object of $\cC$ satisfying the following conditions:
\begin{enumerate}[(i)]
\item $X$ is fixed by $\psi^2$.
\item $X$ and $\lw^2{X}$ are simple.
\item $X$ is not a summand of $\bS_{(2,1)}(X)$.
\end{enumerate}
Then there is a faithful tensor functor $\Phi \colon \cD \to \cC$ satisfying $\Phi(L_{\bb})=X$, which is unique up to isomorphism.
\end{theorem}

\begin{proof}
By Theorem~\ref{thm:key}, $X$ has the structure of an o-Frobenius algebra. By Proposition~\ref{prop:ofrob-to-etale}, $A=X \oplus \bone \oplus X^*$ has the structure of a 1-Delannic algebra. (Note that $X$ does not belong to Case~III, since that would imply $X$ is a summand of $\bS_{(2,1)}(X)$ by Proposition~\ref{prop:case3}.) The universal property of $\cD$ (Theorem~\ref{thm:delannic}) asserts that there is a tensor functor $\Phi \colon \cD \to \cC$ such that $\Phi(\sC(\bR))$ is isomorphic to $A$ as an ordered \'etale algebra. Comparing the simple decompositions, we see that $\Phi(L_{\bb})$ is either $X$ or $X^*$. In the latter case, reversing the order on $A$ yields a functor $\Phi$ with $\Phi(L_{\bb})=X$. Note that reversing the order is equivalent to pre-composing with $\Pi$, and $\Pi(L_{\bb})=L_{\ww}$; see the discussion following Theorem~\ref{thm:delannic} and \cite[Remark~5.1]{line}. The functor $\Phi$ is necessarily faithful since $\cD$ is a semi-simple rigid tensor category.

Suppose $\Phi'$ is a second tensor functor with $\Phi'(L_{\bb})=X$. Then $\Phi(\sC(\bR))$ and $\Phi'(\sC(\bR))$ are two ordered \'etale algebras with underlying object $X \oplus \bone \oplus X^*$. These algebras are isomorphic by Proposition~\ref{prop:ofrob-to-etale}; also, that proposition shows that (under the isomorphism) the orders are either the same or reverse of each other. By Theorem~\ref{thm:delannic}, we see that $\Phi$ is isomorphic to $\Phi'$ or $\Phi' \circ \Pi$. However, the latter is not possible, since $\Phi'(\Pi(L_{\bb}))=X^*$. Thus $\Phi$ and $\Phi'$ are isomorphic, as required.
\end{proof}

We note that if $X$ belongs to Case~I from \S \ref{ss:tri} then it fulfills the conditions of the theorem. In particular, if (i) and (ii) hold and either $X$ is fixed by $\psi^3$ or $\lw^3{X}$ is a non-trivial simple then $X$ belongs to Case~I (Propositions~\ref{prop:psi3-case1} and~\ref{prop:wedge3-case1}), and so the theorem applies.

\begin{corollary} \label{cor:C-1}
For $X$ as above, $\dim{X}=-1$, $X$ has super-exponential growth, and $X$ is fixed by all $\psi^p$ with $p$ different from the characteristic.
\end{corollary}

\begin{proof}
These properties hold for $L_{\bb}$ and are preserved by faithful tensor functors.
\end{proof}

\begin{corollary} \label{cor:no-case-IIa}
In the setting of \S \ref{ss:tri}, Case~IIa cannot occur.
\end{corollary}

\begin{proof}
Suppose $X$ belongs to Case~IIa. Then the hypotheses of Theorem~\ref{thm:C-1} are fulfilled, and so $X$ is fixed by $\psi^3$ by Corollary~\ref{cor:C-1}. But this implies that $X$ belongs to Case~I (Proposition~\ref{prop:psi3-case1}), a contradiction.
\end{proof}

\begin{remark}
Let $X$ and $\Phi$ be as in Theorem~\ref{thm:C-1}. In the Delannoy category, the dual of $L_{\bb}$ is $L_{\ww}$, and we have
\begin{displaymath}
L_{\bb} \otimes L_{\ww} = L_{\bb\ww} \oplus L_{\ww\bb} \oplus L_{\bb} \oplus L_{\ww} \oplus \bone
\end{displaymath}
by \cite[Theorem~7.2]{line}. Applying $\Phi$, we see that (up to relabeling) $Y_1=\Phi(L_{\bb\ww})$ and $Y_2=\Phi(L_{\ww\bb})$, where $Y_1$ and $Y_2$ are as in Proposition~\ref{prop:case1}. In particular, $Y_1=Y_2^*$.
\end{remark}

The following theorem is the second part of Theorem~\ref{thm:C}.

\begin{theorem} \label{thm:C-2}
Let $X$ be as in Theorem~\ref{thm:C-1}, and suppose $X$ generates $\cC$. Then there is an oligomorphic group $(G, \Omega)$ with a regular measure $\mu$ such that $\cC$ is equivalent to the tensor category $\uRep(G, \mu)$. Moreover, $G$ preserves a total order on $\Omega$ and acts transitively on $\Omega^{(n)}$ for $n=1,2,3$.
\end{theorem}

\begin{proof}
By \cite[Corollary~1.4]{discrete}, $\cC$ is equivalent to $\uRep(G, \mu)$ for some pro-oligomorphic group $G$ and regular measure $\mu$. The proof shows that we can take $G$ so that $\bS(G)$ is equivalent to $\Et(\cC)^{\op}$. For a $G$-set $X$, let $A_X$ denote the corresponding \'etale algebra of $\cC$. The measure $\mu$ is uniquely determined by $\mu(A_X)=\dim{A_X}$.

Let $\Omega$ be the $G$-set such that $A_{\Omega}$ is the \'etale algebra $X \oplus \bone \oplus X^*$. It follows from the results in \S \ref{ss:ofrob-to-etale} (and in particular Lemma~ \ref{lem:ofrob-to-etale-5}) that $G$ preserves a total order on $\Omega$, and acts transitively on $\Omega^{(n)}$ for $n=1,2,3$. Note that $G$ has finitely many orbits on $\Omega^n$ for all $n \ge 0$ since $\bS(G)$ is closed under finite products \cite[\S 2.3]{repst}.

We have proved the theorem, except with a pro-oligomorphic group instead of an oligomorphic group. We now explain how to actually get an oligomorphic group. Let $N$ be the normal subgroup of $G$ that acts trivially on $\Omega$, and let $G'=G/N$. Then $(G', \Omega)$ is oligomorphic. The measure $\mu$ induces a measure $\mu'$ on $G'$, and $\uRep(G', \mu')$ is a tensor subcategory of $\uRep(G, \mu)$. The resulting functor $\uRep(G', \mu') \to \cC$ is clearly fully faithful, and it is also essentially surjective since $A_{\Omega}$ belongs to its image, and is a generator for $\cC$.
\end{proof}

We also prove one additional result here. Recall $\HH$ is the non-abelian group of order~21 and $V$ is an irreducible three dimensional representation of it (see \S \ref{ss:case3}).

\begin{theorem} \label{thm:mod}
Suppose that $X$ is a generator of $\cC$ satisfying the following conditions:
\begin{enumerate}[(i)]
\item $X$ is fixed by $\psi^2$.
\item $X$ and $\lw^2{X}$ are simple.
\item $X$ has moderate growth.
\end{enumerate}
Then there is an equivalence $\Phi \colon \Rep(\HH) \to \cC$ such that $\Phi(V)=X$.
\end{theorem}

\begin{proof}
The object $X$ belogns to one of the three cases from \S \ref{ss:tri}. In Case~I, $X$ is not a summand of $\bS_{(2,1)}(X)$ (Proposition~\ref{prop:case1}), and so Corollary~\ref{cor:C-1} shows that $X$ has super-exponential growth, a contradiction. In Case~II, we have an isomorphism
\begin{displaymath}
X \otimes X^* = \lw^2{X} \oplus \lw^2{X^*} \oplus X \oplus X^* \oplus \bone
\end{displaymath}
by Proposition~\ref{prop:case2}(b). However, since $X$ has moderate growth, we can regard $X$ as a representation of a finite group scheme (Proposition~\ref{prop:moderate}), and this gives a contraction: the above isomorphism would show that the vector space dimension of $X$ is $-1$. We must therefore belong to Case~III, and so the result follows from Proposition~\ref{prop:case3}.
\end{proof}

\subsection{Theorem~\ref{thm:A}}

We begin with a general observation about functors from $\cD$. 

\begin{proposition} \label{prop:del-full}
Let $\Phi \colon \cD \to \cC$ be a tensor functor, and put $X=\Phi(L_{\bb})$. Suppose all exterior powers of $X$ are simple. Then these simples are pairwise non-isomorphic and $\Phi$ is fully faithful.
\end{proposition}

\begin{proof}
We first recall two key facts about the Delannoy category relevant to this proof. First, the exterior powers of $L_{\bb}$ are distinct simple objects \cite[Proposition~8.5]{line}. And second, we have a decomposition
decomposition
\begin{displaymath}
L_{\bb}^{\otimes n} = \bigoplus_{r=0}^n (\lw^r{L_{\bb}})^{\oplus c(n,r)}
\end{displaymath}
for some multiplicities $c(n,r)$ by \cite[Theorem~7.2]{line} (or \cite[Corollary~8.11]{line} in characteristic~0). Applying $\Phi$, we obtain a similar decomposition of tensor powers of $X$.

We now show that the exterior powers of $X$ are pairwise non-isomorphic. Suppose, by way of contradiction, that $\lw^n{X}$ is isomorphic to $\lw^m{X}$ for some $n>m$. Let $\Sigma$ be the class of all objects in $\cC$ whose simple constituents are among $\lw^i{X}$ for $0 \le i<n$. Note that this contains $X^{\otimes i}$ for $0 \le i \le n$ by making use of the above decomposition, and the isomorphism $\lw^n{X}=\lw^m{X}$ in the case $i=n$. It follows that $\Sigma$ is closed under tensoring by $X$: indeed, if $0 \le i <n$ then $X \otimes \lw^i{X}$ is a subquotient of $X^{\otimes (i+1)}$, which we have just seen belongs to $\Sigma$. Since $\Sigma$ contains $\bone$, it thus contains $X^{\otimes i}$ for all $i$. We therefore see that only finitely many simples appear in tensor powers of $X$. This implies that $X$ has moderate growth, which is a contradiction (since $\Phi$ is necessarily faithful).

We now prove that $\Phi$ is full. Put
\begin{displaymath}
L^{n,m} = L_{\bb}^{\otimes n} \otimes L_{\ww}^{\otimes m}, \qquad
X^{n,m} = X^{\otimes n} \otimes (X^{\vee})^{\otimes m}.
\end{displaymath}
We have
\begin{displaymath}
\dim \Hom(L^{n,0}, L^{m,0}) = \sum_{r=0}^{\min(n,m)} c(n,r) c(m,r),
\end{displaymath}
where $c$ is as above. Since the simple decomposition of $X^{n,0}$ matches that of $L^{n,0}$, we obtain exactly the same formula for the dimension of $\Hom(X^{n,0}, X^{m,0})$. Next, observe that
\begin{displaymath}
\Hom(L^{a,b}, L^{c,d}) = \Hom(L^{a+d,0}, L^{b+c,0})
\end{displaymath}
by adjunction. There is a similar formula using $X$'s. We thus see that
\begin{displaymath}
\Phi \colon \Hom(L^{a,b}, L^{c,d}) \to \Hom(X^{a,b}, X^{c,d})
\end{displaymath}
is an isomorphism, since it is injective and the source and target have the same dimension. Since every simple of $\cD$ is a summand of some $L^{a,b}$, the result follows.
\end{proof}

\begin{proof}[Proof of Theorem~\ref{thm:A}]
Let $X$ be a generator for $\cC$ such that $X$ is fixed by $\psi^2$ and all exterior powers of $X$ are simple. By Proposition~\ref{prop:wedge3-case1}, $X$ belongs to Case~I from \S \ref{ss:tri}; note that $\lw^3{X}$ is not the tensor unit since then $\lw^4{X}$ would vanish\footnote{Alternatively, the proof of Proposition~\ref{prop:wedge3-case1} shows we are in Case~I or~III, and we are clearly not in~III.} \cite[Proposition~2.3.2]{CEN}. Thus $X$ does not appear in $\bS_{(2,1)}(X)$ (Proposition~\ref{prop:case1}). Applying Theorem~\ref{thm:C-1}, there is a tensor functor $\Phi \colon \cD \to \cC$ such that $\Phi(L_{\bb})=X$, which is unique up to isomorphism. By Proposition~\ref{prop:del-full}, $\Phi$ is fully faithful. Since $X$ is a generator for $\cC$, it follows that $\Phi$ is an equivalence.
\end{proof}

\subsection{Theorem~\ref{thm:B}} \label{ss:thmB}

Recall that $\rK_+(\cC)$ is the Grothendieck semi-ring of $\cC$. This is the sub-semi-ring of $\rK(\cC)$ consisting of effective classes. As with any semi-group, $\rK_+(\cC)$ carries a canonical order, via $y \le x$ if $x=y+z$ for some $z \in \rK_+(\cC)$. For objects $X$ and $Y$ of $\cC$, we have $[Y] \le [X]$ if and only if there is an isomorphism $X \cong Y \oplus Z$ for some object $Z$. In particular, $X$ is simple if and only if $[X]$ is a minimal non-zero element of $\rK_+(\cC)$.

We now show that $\cD$ is determined by its Grothendieck semi-ring.

\begin{theorem} \label{thm:recog}
Suppose $i \colon \rK_+(\cD) \to \rK_+(\cC)$ is a semi-ring isomorphism. Then there is an equivalence $\Phi \colon \cD \to \cC$ of tensor categories.
\end{theorem}

Note that in this theorem we are not claiming any relation between $\Phi$ and $i$. In a subsequent result, we show that $\Phi$ can be chosen to induce $i$.

\begin{proof}
Let $[X]=i([L_{\bb}])$, i.e., $X$ is an object of $\cC$ whose class is $i([L_{\bb}])$. More generally, let $L_n=\lw^n{L_{\bb}}$, and put $[X_n]=i([L_n])$. The $L_n$'s are pairwise non-isomorphic simple objects of $\cD$ \cite[Proposition~8.5]{line}. Since $i$ is a semi-ring isomorphism, we see that the $X_n$'s are pairwise non-isomorphic simple objects of $\cC$. Here we are using the observation discussed above, that an object is simple if and only if its class is a minimal non-zero element.

In $\rK(\cD) \otimes \bQ$, we have the identity
\begin{displaymath}
[L_n] = \binom{[L_1]}{n}.
\end{displaymath}
In characteristic~0, this follows from the fact that the Adams operations are trivial on $\rK(\cD)$; indeed, this implies that $\rK(\cD)$ is a binomial ring and the $\lambda$-operations are given by binomial coefficients \cite[Proposition~8.3]{Elliott}. The positive characteristic case can be deduced from the characteristic~0 case, since this is just a computation in the ring $\rK(\cD)$, which is independent of the base field (by \cite[Theorem~7.2]{line}); here we are also using the explicit calculation of $\lw^n{L_{\bb}}$ \cite[Proposition~8.5]{line}.

The previous paragraph implies that we have similar identities for the $X_n$'s. In particular, letting $d=\dim(X)$, we find
\begin{displaymath}
\dim(X_n)=\binom{d}{n}
\end{displaymath}
whenever $n!$ is invertible in $k$; since the characteristic is not 2 or 3, this holds for $n \le 4$.

In $\rK(\cC)$, we have
\begin{displaymath}
X \otimes X = i(L_1 \otimes L_1) = i(L_1 \oplus L_2^{\oplus 2}) = X \oplus X_2^{\oplus 2}.
\end{displaymath}
If $\Sym^2{X}$ or $\lw^2{X}$ vanished then $X$ would be invertible \cite[Proposition~2.3.2]{CEN}, and so $X \otimes X$ would be invertible, but this is a contradiction since invertible objects are simple. Thus both $\Sym^2{X}$ and $\lw^2{X}$ are non-zero. We now exclude more possibilities for $\Sym^2{X}$. Note that the dimension of a simple object in $\cC$ is non-zero; in particular, $d$ is non-zero.
\begin{itemize}
\item If $\Sym^2(X)=X_2^{\oplus 2}$ then $\tfrac{1}{2} d(d+1)=d(d-1)$, and so $d=3$. But then $\dim(X_4)=0$, a contradiction.
\item If $\Sym^2(X)=X_2$ then $\tfrac{1}{2} d(d+1)=\tfrac{1}{2} d(d-1)$, a contradiction.
\item If $\Sym^2(X)=X$ then $\tfrac{1}{2} d(d+1)=d$, and so $d=1$. But then $\dim(X_2)=0$, a contradiction.
\end{itemize}
We conclude that $\Sym^2{X}=X \oplus X_2$ and $\lw^2{X}=X_2$. In particular, $X$ is fixed by $\psi^2$ and $\lw^2{X}$ is simple.

Since $L_{\bb}^{\otimes 3}$ contains $L_{\bb}$ with multiplicity one (e.g., by \cite[Theorem~7.2]{line}), it follows that $X^{\otimes 3}$ contains $X$ with multiplicity one. Since $X$ is fixed by $\psi^2$, the unique copy of $X$ in $X^{\otimes 3}$ is contained in $\Sym^3{X}$ by Proposition~\ref{prop:sym3}. In particular, $X$ is not a summand of $\bS_{(2,1)}(X)$.

Applying Theorem~\ref{thm:C}, we have a faithful tensor functor $\Phi \colon \cD \to \cC$ satisfying $\Phi(L_{\bb})=X$. Since $\Phi$ is a tensor functor, we have $\Phi(L_n)=\lw^n{X}$. Consider the composition
\begin{displaymath}
\xymatrix{
\rK(\cD) \ar[r]^{\Phi} & \rK(\cC) \ar[r]^{i^{-1}} & \rK(\cD) }
\end{displaymath}
This is a ring homomorphism that takes $[L_1]$ to $[L_1]$. It follows that $[L_n]$ is mapped to $[L_n]$, since $[L_n]$ is a polynomial in $[L_1]$. We thus see that the class of $\Phi(L_n)=\lw^n{X}$ maps to $[L_n]$ under $i^{-1}$, and is thus $[X_n]$. We have therefore shown $\lw^n{X}=X_n$; essentially, we have verified that $i$ is compatible with the $\lambda$-ring operations, in this special case. In particular, $\lw^n{X}$ is simple for all $n$, and so $\Phi$ is full (Proposition~\ref{prop:del-full}).

Now, $L_{\ww}$ is the unique simple object of $\cD$ such that $1 \le [L_{\ww}] \cdot [L_{\bb}]$ in $\rK_+(\cD)$. Similarly, $X^*$ is the unique simple object of $\cC$ such that $1 \le [X^*] \cdot [X]$ in $\rK_+(\cC)$. Since $i$ is a semi-ring isomorphism, we have $i([L_{\ww}])=[X^*]$. Let $M$ be a simple object of $\cC$, and let $[M']=i^{-1}([M])$, so that $M'$ is a simply object of $\cD$. Since $L_{\bb}$ generates $\cD$, we see that $M'$ is contained in $L_{\bb}^{\otimes n} \otimes L_{\ww}^{\otimes m}$ for some $n$ and $m$. This implies that $[M'] \le [L_{\bb}]^n \cdot [L_{\ww}]^m$ in $\rK_+(\cD)$. Applying $i$, we find $[M] \le [X]^n \cdot [X^*]^m$ in $\rK_+(\cC)$, and so $M$ appears in $X^{\otimes n} \otimes (X^*)^{\otimes m}$. This shows that $X$ generates $\cC$, and so $\Phi$ is an equivalence.
\end{proof}

We have just seen that if a semi-ring isomorphism $i \colon \rK_+(\cD) \to \rK_+(\cC)$ exists then an equivalence $\Phi \colon \cD \to \cC$ exists. We now show that there is actually a natural bijection between $i$'s and $\Phi$'s. To this end, we first introduce some notation. Let $\Eq^0(\cD, \cC)$ (resp.\ $\Eq^1(\cD, \cC)$) denote the set of all isomorphism classes of equivalences $\cD \to \cC$ (resp.\ $\cD^{\op} \to \cC$), and let $\Eq^*(\cD, \cC)$ be the disjoint union of the $\Eq^i(\cD, \cC)$. Let $\Isom(\rK_+(\cD), \rK_+(\cC))$ denote the set of isomorphism classes of semi-rings. The following is the result we aim to prove:

\begin{theorem} \label{thm:recog2}
The natural map
\begin{equation} \label{eq:eq}
\Eq^*(\cD, \cC) \to \Isom(\rK_+(\cD), \rK_+(\cC))
\end{equation}
is a bijection. Moreover, if $\cC \cong \cD$ then both sides have cardinality four.
\end{theorem}

If either the domain or target of \eqref{eq:eq} is non-empty then $\cC$ and $\cD$ are equivalent; here we are using Theorem~\ref{thm:recog}. It thus suffices to prove the result when $\cC=\cD$, and we assume this in what follows. In this case, \eqref{eq:eq} is a group homomorphism.

We begin by analyzing the group $\Eq^*(\cD, \cD)$. The set $\Eq^1(\cD, \cD)$ is non-empty, since it contains the duality functor. It follows that $\Eq^0(\cD, \cD)$ is an index two subgroup of $\Eq^*(\cD, \cD)$. Let $\Pi \colon \cD \to \cD$ be the auto-equivalence discussed after Theorem~\ref{thm:delannic}. As explained in \cite[Remark~4.17]{line}, we have $\Pi(L_{\bb})=L_{\ww}$, and, in general, $\Pi(L_{\lambda})=L_{\mu}$, where $\mu=\rev(\lambda^{\vee})$; here $\rev$ means we reverse the word, and $(-)^{\vee}$ means we switch $\bb$ and $\ww$. This equivalence squares to the identity. We thus have an order two subgroup of $\Eq^0(\cD, \cD)$.

\begin{lemma} \label{lem:recog2-1}
The map \eqref{eq:eq} is injective.
\end{lemma}

\begin{proof}
Suppose $\Phi$ belongs to the kernel, that is, $\Phi$ is a (possibly contravariant) auto-equivalence of $\cD$ as a tensor category that induces the identity on $\rK_+(\cD)$. First suppose that $\Phi$ is covariant. Then $\Phi(\sC(\bR))$ is a 1-Delannic algebra in $\cD$ whose underlying object is isomorphic to $\sC(\bR)$. From Proposition~\ref{prop:ofrob-to-etale} we see that $\Phi(\sC(\bR))$ is isomorphic, as an ordered \'etale algebra, to $\sC(\bR)$ with its standard order or the reverse order. Under the bijection of Theorem~\ref{thm:delannic}, these two Delannic algebras correspond to the identity functor and $\Pi$. Thus, by that theorem, $\Phi$ is isomorphic to either the identity or $\Pi$. But it cannot be $\Pi$ since $\Pi$ is not in the kernel of \eqref{eq:eq}.

Now suppose $\Phi$ is contravariant. Let $\Phi'$ be the composition of $\Phi$ with duality. Then $\Phi'$ is a covariant auto-equivalence of $\cD$ and induces the duality map on $\rK_+(\cD)$. Since $\sC(\bR)$ is self-dual, we again see that $\Phi'(\sC(\bR))$ is a 1-Delannic algebra with underlying object $\sC(\bR)$. As in the previous paragraph, $\Phi'$ is either the identity or $\Pi$. However, neither of these functors induce the duality map on $\rK_+(\cD)$, and so this case cannot happen.
\end{proof}

Given a weight $\lambda \in \Lambda$, let $S_{\bb}(\lambda)$ denote the multiset of weights obtained by deleting one $\bb$ from $\lambda$ in each possible way. Analogously define $S_{\ww}(\lambda)$. For example,
\begin{displaymath}
S_{\bb}(\bb\bb\ww\bb) = \{ \bb\ww\bb, \bb\ww\bb, \bb\bb\ww \}.
\end{displaymath}

\begin{lemma} \label{lem:recog2-2}
Let $\lambda$ and $\mu$ be weights of lengths $\ge 3$. If $S_{\bb}(\lambda)=S_{\bb}(\mu)$ and $S_{\ww}(\lambda)=S_{\ww}(\mu)$ then $\lambda=\mu$.
\end{lemma}

\begin{proof}
Let $\ell$ be the length of $\lambda$. Every weight in $S_{\bb}(\lambda)$ has length $\ell-1$. Thus the same is true for $S_{\bb}(\mu)$, and so $\mu$ also has length $\ell$. Without loss of generality, suppose that $\lambda$ begins with $\bb$. If $\lambda$ begins with $\bb\bb$ then every word in $S_{\bb}(\lambda)$ begins with $\bb$; if $\lambda$ begins with $\bb\ww$ then there is one word in $S_{\bb}(\lambda)$ that begins with $\ww$, and every other one begins with $\bb$. Since $S_{\bb}(\mu)=S_{\bb}(\lambda)$, the same is true for $S_{\bb}(\mu)$. From this, we conclude that $\mu$ also begins with $\bb$. Note that we have used the assumption $\ell \ge 3$ here.

Observe that $\lambda$ consists of only $\bb$ characters if and only if the same is true for every weight in $S_{\bb}(\lambda)$. Thus $\lambda$ is all $\bb$ if and only if $\mu$ is, and in this case $\lambda=\mu$. We assume in what follows that $\lambda$ and $\mu$ both contain some $\ww$ character.

Let $r+1$ be the index of the first $\ww$ character in $\lambda$, so that the first $r$ characters of $\lambda$ are $\bb$. Write $\lambda=\bb \lambda'$; note that the first index of $\ww$ in $\lambda'$ is $r$. If we remove one of the first $r$ characters from $\lambda$, we obtain $\lambda'$; if we remove one of the other characters from $\lambda$, we obtain a word that begins with $r$ copies of $\bb$. We thus see that $S_{\bb}(\lambda)$ contains $\lambda'$ with multiplicity $r$, and all remaining words begin with $r$ copies of $\bb$. We can thus recover $\lambda'$ as the unique weight (ignoring multiplicity) in $S_{\bb}(\lambda)$ which has the fewest initial copies of $\bb$. Writing $\mu=\bb \mu'$, we find a similar description of $\mu'$. From the equality $S_{\bb}(\lambda)=S_{\bb}(\mu)$, we conclude that $\lambda'=\mu'$, and so $\lambda=\mu$, as required.
\end{proof}

Write $a_{\lambda}$ for the class of $L_{\lambda}$ in $\rK_+(\cD)$. We call these the \defn{simple classes}. They are characterized intrinsically as the minimal non-zero elements of $\rK_+(\cD)$ with respect to its natural partial order. The multiplicity of $\mu$ in $S_{\bb}(\lambda)$ is exactly the coefficient of $\lambda$ in $a_{\bb} a_{\mu}$; this follows from the fusion rule for $\cD$ \cite[Theorem~7.2]{line}

\begin{lemma} \label{lem:recog2-3}
Let $i \colon \rK_+(\cD) \to \rK_+(\cD)$ be a semi-ring automorphism that fixes each $a_{\lambda}$ with $\ell(\lambda) \le 2$. Then $i$ is the identity.
\end{lemma}

\begin{proof}
We show by induction on $\ell$ that $i$ fixes all $a_{\lambda}$ with $\lambda$ of length $\ell$. We are given this for $\ell \le 2$. Suppose now that $\ell \ge 3$, and we know the statement for all smaller $\ell$. Let $\lambda$ be a weight of length $\ell$. Since $a_{\lambda}$ a simple class, so is $i(a_{\lambda})$, and it thus has the form $a_{\mu}$. Since $i$ fixes all weights of smaller length, $\mu$ must have length at least $\ell$.

Let $\nu$ have length $\ell-1$. The multiplicity of $\nu$ in $S_{\bb}(\lambda)$ is the coefficient of $a_{\lambda}$ in $a_{\bb} a_{\nu}$. Similarly, the multiplicity of $\nu$ in $S_{\bb}(\mu)$ is the coefficient of $a_{\mu}=i(a_{\lambda})$ in $a_{\bb} a_{\nu}=i(a_{\bb} a_{\nu})$. Since $i$ is an isomorphism, these quantities are equal. Since $S_{\bb}(\lambda)$ contains some $\nu$ of length $\ell-1$, we see that the same is true for $S_{\bb}(\mu)$, and so $\mu$ has length $\ell$. It now follows that $S_{\bb}(\lambda)=S_{\bb}(\mu)$, and similarly for $S_{\ww}$. Thus $\lambda=\mu$ by Lemma~\ref{lem:recog2-2}, and the proof is complete.
\end{proof}

We are now ready to finish the proof of the theorem.

\begin{proof}[Proof of Theorem~\ref{thm:recog2}]
We show that \eqref{eq:eq} is surjective. Thus let $i$ be a semi-ring automorphism of $\rK_+(\cD)$. Since $a_{\bb}^2=a_{\bb}+2a_{\bb\bb}$, we see that $i(a_{\bb})$ is a simple class whose square is a sum of three simple classes. From the fusion rules for $\cD$ \cite[Theorem~7.2]{line}, it follows that $i(a_{\bb})$ is either $a_{\bb}$ or $a_{\ww}$. In the latter case, we can modify $i$ by an element of the image of \eqref{eq:eq} (namely, the image of the duality functor) to put us in the first case. We thus assume that $i(a_{\bb})=a_{\bb}$. Since $i(a_{\ww})$ must also be a simple class of length~1, we have $i(a_{\ww})=a_{\ww}$.

We have $a_{\bb}^2=a_{\bb}+2a_{\bb\bb}$. In particular, $a_{\bb\bb} \le a_{\bb}^2$. Applying $i$, we see that $i(a_{\bb\bb})$ is a simple class that is $\le a_{\bb}^2$. Since it cannot be $a_{\bb}$, it must be $a_{\bb\bb}$. Thus $a_{\bb\bb}$ is fixed by $i$. Similarly, so is $a_{\ww\ww}$. We have
\begin{displaymath}
a_{\bb} a_{\ww} = a_{\bb\ww}+a_{\ww\bb}+a_{\bb}+a_{\ww}+1.
\end{displaymath}
Thus $a_{\bb\ww} \le a_{\bb} a_{\ww}$, and so $i(a_{\bb\ww}) \le a_{\bb}a_{\ww}$. It follows that $i(a_{\bb\ww})$ is either $a_{\bb\ww}$ or $a_{\ww\bb}$. Once again, in the latter case we can modify $i$ by an element of the image of \eqref{eq:eq} (namely, duality composed with $\Pi$) to place ourselves in the first case (while maintaining the behavior on words of length one). We thus assume $i(a_{\bb\ww})=a_{\bb\ww}$. It now follows that $i(a_{\ww\bb})=a_{\ww\bb}$. Since $i$ fixes all simple classes of length $\le 2$, it is the identity by Lemma~\ref{lem:recog2-3}. Thus $i$ belongs to the image of \eqref{eq:eq}.

We have thus shown that \eqref{eq:eq} is surjective. Since it is also injective (Lemma~\ref{lem:recog2-1}), it is an isomorphism. The above proof shows that every automorphism of $\rK_+(\cD)$ is the image of one of the four elements of $\Eq^*(\cD, \cD)$ that we have constructed. It thus follows that the source and target of \eqref{eq:eq} have order four.
\end{proof}

\section{Some comments on Case~IIb} \label{s:caseIIb}

Suppose $X$ is an object in a semi-simple pre-Tannakian category $\cC$ such that $X$ is fixed by $\psi^2$, and $X$ and $\lw^2{X}$ are simple. Then $X$ belongs to one of the cases from \S \ref{ss:tri}. In Case~I, $X$ comes from the Delannoy category, and in Case~III it comes from $\Rep(\HH)$; furthermore, Case~IIa has been excluded. It remains to understand Case~IIb. We now make some comments on it, though we do not resolve it. We assume throughout \S \ref{s:caseIIb} that $X$ belongs to Case~IIb.

\subsection{An attempt at an o-Frobenius structure}

We first attempt to define an o-Frobenius structure on $X$, following \S \ref{ss:key}. There are non-zero maps
\begin{displaymath}
\mu \colon X \otimes X \to X, \qquad \delta \colon X \to X \otimes X,
\end{displaymath}
which are unique up to scaling. We fix such maps satisfying $\mu\delta=\id_X$. These maps are necessarily (co-)commutative. We thus have axioms (a) and (c) for an o-Frobenius algebra. Previously, (co-)associativity, which is axiom~(b), came from the assumption that $X$ does not appear in $\bS_{(2,1)}$. We can now no longer reason in this manner, and, in fact, we will not obtain a proof of (co-)associativity. We set this issue aside for the moment and examine the remaining axiom. As usual, define
\begin{displaymath}
\gamma = (\id \otimes \mu)(\delta \otimes \id), \qquad
\gamma' = (\mu \otimes \id)(\id \otimes \delta).
\end{displaymath}

\begin{proposition} \label{prop:IIb-axiom-d}
We have $\gamma+\gamma'=\id_{X \otimes X} + \delta \mu$, i.e., axiom (d) holds.
\end{proposition}

\begin{proof}
The proof is a modification of the argument used for Lemma~\ref{lem:key-4}. We have
\begin{displaymath}
\tr(\gamma)=\tr(\gamma^2)=0, \qquad
\tr(\gamma \tau)=\tr((\gamma \tau)^2)=\dim{X}.
\end{displaymath}
We have already explained how to compute the traces of $\gamma$ and $\gamma \tau$ in Proposition~\ref{prop:trace-gamma}, and their squares are handled in a similar manner. Since we are in Case~IIb, we necessarily have $\dim{X}=-1$ (Proposition~\ref{prop:case2}(b)).

Let $R=\End(X^{\otimes 2})$, and let $e=\tfrac{1}{2}(1+\tau)$ and $f=\tfrac{1}{2}(1-\tau)$. As in Lemma~\ref{lem:key-4}, the ring $fRf$ is one dimensional and spanned by $f$, and so we have $f\gamma f=cf$ for some $c$. Computing traces, as in that lemma, we find $c=\tfrac{1}{2}$. The ring $eRe$ is two-dimensional, and spanned by $\delta \mu$ and $e-\delta \mu$. We thus have an expression
\begin{displaymath}
e\gamma e = \tfrac{1}{2} c_1 (e-\delta \mu) + \tfrac{1}{2} c_2 \delta \mu.
\end{displaymath}
Computing traces of the above equation, as in Lemma~\ref{lem:key-4}, gives $-1=c_1-c_2$. Note that in the proof of that lemma, we knew $\gamma \delta \mu=\delta \mu$, and this allowed us to eliminate one degree of freedom; we cannot do that here.

Adding our equations for $e\gamma e$ and $f\gamma f$, we find
\begin{displaymath}
\gamma+\gamma' = f + c_1(e-\delta \mu) + c_2 \delta \mu.
\end{displaymath}
Now square both sides. Since $f$, $e-\delta \mu$, and $\delta \mu$ are orthogonal idempotents, we find
\begin{displaymath}
\gamma^2+\gamma \gamma' + \gamma' \gamma + (\gamma')^2 = f + c_1^2 (e-\delta \mu) + c_2^2 \delta \mu.
\end{displaymath}
Now take traces. Note that $\gamma$ and $\gamma'$ are conjugate, and $\gamma \gamma' = (\gamma \tau)^2$. Thus, using the formulas from the first paragraph, we find
\begin{displaymath}
-2 = 1 + c_1^2 - c_2^2.
\end{displaymath}
Combined with our other equation, this gives $c_1=1$ and $c_2=2$, which proves the proposition.
\end{proof}

\begin{corollary} \label{cor:IIb-axiom-d}
We have $\mu \gamma = \mu$ and $\gamma \delta = \delta$.
\end{corollary}

\begin{proof}
Since any map $X^{\otimes 2} \to X$ is invariant under $\tau$, it follows that $\mu \gamma=\mu \gamma'$. Thus, from axiom~(d), we find
\begin{displaymath}
2 \mu \gamma = \mu + \mu \delta \mu = 2 \mu,
\end{displaymath}
and so $\mu \gamma = \mu$. The computation of $\gamma \delta$ is similar.
\end{proof}

While we cannot prove axiom~(b) in full, we do get exactly half of it:

\begin{proposition} \label{prop:IIb-assoc}
Either $\mu$ is associative or $\delta$ is co-associative, but not both.
\end{proposition}

\begin{proof}
First suppose that $\mu$ is associative and $\delta$ is co-associative. Then $X$ is an o-Frobenius algebra. Just as in the proof of Theorem~\ref{thm:C-1}, we obtain a tensor functor $\Phi \colon \cD \to \cC$ satisfying $\Phi(L_{\bb})=X$. But this implies that $X$ is fixed by $\psi^3$, and so $X$ belongs to Case~I (Proposition~\ref{prop:psi3-case1}), a contradiction.

Now suppose that $\mu$ is not associative; we will show $\delta$ is co-associative. From Proposition~\ref{prop:sym3}, we find
\begin{displaymath}
X^{\otimes 3} = \bS_{(2,1)}(X)^{\oplus 2} \oplus (\lw^3{X})^{\oplus 2} \oplus (\lw^2{X})^{\oplus 2} \oplus X,
\end{displaymath}
and so $X$ appears with multiplicity three in $X^{\otimes 3}$; note that $X$ appears with multiplicity one in $\bS_{(2,1)}(X)$ and multiplicity~0 in $\lw^3{X}$ by Proposition~\ref{prop:case2}(c). For $1 \le i \le 3$, let
\begin{displaymath}
\mu_i \colon X^{\otimes 3} \to X
\end{displaymath}
be the map where we first multiply the two factors different from $i$, and then multiply the result with $i$. Define $\delta_i$ analogously. If we take into account the $\fS_3$ action on the  decomposition of $X^{\otimes 3}$, we see that $\Hom(X^{\otimes 3}, X)$ is the permutation representation of $\fS_3$. The $\mu_i$'s form an $\fS_3$-orbit, are not fixed by $\fS_3$ (since $\mu$ is not associative), and have non-zero restriction to $\Sym^3(X)$ (since $\mu_1\delta_1=\id_X$). It follows that the $\mu_i$'s form a basis of $\Hom(X^{\otimes 3}, X)$.

The natural map
\begin{displaymath}
\langle, \rangle \colon \Hom(X, X^{\otimes 3}) \times \Hom(X^{\otimes 3}, X) \to \End(X) = k
\end{displaymath}
is a perfect pairing. As remarked, $\mu_1 \delta_1=\id_X$, and so $\langle \delta_1, \mu_1 \rangle=1$. We have
\begin{displaymath}
\mu_1 \delta_3 = \mu \gamma \delta = \mu \delta = \id_X,
\end{displaymath}
where the first step is tautological and the second uses Corollary~\ref{cor:IIb-axiom-d}. Thus $\langle \delta_3, \mu_1 \rangle = 1$. All other cases are similar to one of these two, and so we find $\langle \delta_j, \mu_i \rangle = 1$ for all $i$ and $j$. Since this is a perfect pairing and the $\mu_i$'s are a basis, we find that $\delta_1=\delta_2=\delta_3$, which shows that $\delta$ is co-associative.

The case where $\delta$ is not co-associative is similar; in fact, we can simply reduce to the above case by considering $X^*$.
\end{proof}

\subsection{Further calculations}

Suppose we are in the above situation, where $\mu$ is not associative and therefore $\delta$ is co-associative. Let $\mu_1$, $\mu_2$, and $\mu_3$ be as in the proof of Proposition~\ref{prop:IIb-assoc}. There is one additional computation that we wish to explain.

\begin{proposition} \label{cor:IIb-compositions}
Letting $\kappa=\sqrt{-3}$, we have
\begin{align*}
\mu_1 \circ (\gamma' \otimes \id) &= \tfrac{1}{4} (2-\kappa-\kappa^{-1}) \mu_1 + \tfrac{1}{4} (\kappa-\kappa^{-1}) \mu_2 + \tfrac{1}{2} (1+\kappa^{-1}) \mu_3 \\
\mu_2 \circ (\gamma' \otimes \id) &= \tfrac{1}{4} (\kappa^{-1}-\kappa) \mu_1 + \tfrac{1}{4} (2+\kappa+\kappa^{-1}) \mu_2 + \tfrac{1}{2} (1-\kappa^{-1}) \mu_3 \\
\mu_3 \circ (\gamma' \otimes \id) &= \mu_3.
\end{align*}
\end{proposition}

\begin{proof}
Let $M=\Hom(X^{\otimes 3}, X)$, which is a three dimensional $k$-vector space with basis $\mu_1$, $\mu_2$, $\mu_3$, and let $R=\End(X^{\otimes 2})$. We regard $M$ as a right $R$-module via $x \cdot a = x \circ (a \otimes 1)$, where $x \in M$ and $a \in R$.

The ring $R$ is isomorphic to $\rM_2(k) \oplus \rM_1(k)$, and thus has two simple modules $L_1=k$ and $L_2=k^2$. The element $\delta \mu$ of $R$ acts by the identity on $L_1$ and by~0 on $L_2$. A simple computation (similar to what appeared in the proof of Proposition~\ref{prop:IIb-assoc}) shows that $\mu_i \cdot \delta\mu=\mu_3$ for all $i$. Thus $\delta \mu$ is not the identity on $M$, and not identically~0, and so we must have an isomorphism $M \cong L_1 \oplus L_2$. More canonically, we have
\begin{displaymath}
M=M_1 \oplus M_2, \qquad M_1=\im(\delta \mu), \qquad M_2=\ker(\delta \mu),
\end{displaymath}
and $M_1 \cong L_1$ and $M_2 \cong L_2$. Clearly, $M_1$ is spanned by $\mu_3$, and $M_2$ has a basis given by $\mu_1-\mu_3$ and $\mu_2-\mu_3$.

From the results obtained so far, one can show (see Remark \ref{rem:AssocNotNeeded}) that there is a unique algebra homomorphism $\Phi \colon R \to \rM_2(k)$ satisfying
\begin{displaymath}
\Phi(\tau) = \begin{pmatrix} 1 \\ & -1 \end{pmatrix}, \qquad
\Phi(\gamma) = \tfrac{1}{2} \begin{pmatrix} 1 & 1 \\ 1 & 1 \end{pmatrix}, \qquad
\Phi(\gamma') = \tfrac{1}{2} \begin{pmatrix} 1 & -1 \\ -1 & 1 \end{pmatrix}.
\end{displaymath}
We think of $L_2=k^2$ as the space of row vectors, with $R$ acting on the right by matrix multiplication. We let $e_1$ and $e_2$ be the standard basis of $L_2$.

Now, in $L_2$ the vectors $e_1$ and $e_2$ span the $+1$ and $-1$ eigenspaces of $\tau$. In $M_2$, these eigenspaces are spanned by
\begin{displaymath}
\mu_1+\mu_2-2\mu_3 \qquad \text{and} \qquad \mu_1-\mu_2
\end{displaymath}
respectively. It follows that there is a unique $R$-module isomorphism $i \colon L_2 \to M_2$ satisfying
\begin{displaymath}
i(e_1) = \mu_1+\mu_2-2 \mu_3.
\end{displaymath}
Since $i(e_2)$ spans the $-1$ eigenspace of $\tau$, we must have
\begin{displaymath}
i(e_2) = \kappa (\mu_1-\mu_2)
\end{displaymath}
for some non-zero scalar $\kappa$.

Now, our present goal is to compute $\mu_i \cdot \gamma'$. For $i=3$, the stated formula follows directly from Corollary~\ref{cor:IIb-axiom-d}. The formula for $\Phi(\gamma')$ means
\begin{displaymath}
e_1 \gamma' = \tfrac{1}{2} (e_1-e_2), \qquad e_2 \gamma' = \tfrac{1}{2} (e_2-e_1).
\end{displaymath}
Applying $i$ to these equations and doing some basic linear algebra gives the formulas for $\mu_1 \cdot \gamma'$ and $\mu_2 \cdot \gamma'$.

Finally, we must compute $\kappa$. Writing down the diagram for $\mu_1 \cdot \gamma'$, we see that it is invariant under the transposition $(1\;3)$. The space of $(1\;3)$ invariant vectors in $M$ is spanned by $\mu_2$ and $\mu_1+\mu_3$. We thus have a relation
\begin{displaymath}
\mu_1 \cdot \gamma' = \lambda \mu_2 + \lambda'(\mu_1+\mu_3)
\end{displaymath}
for scalars $\lambda$ and $\lambda'$. In particular, the coefficients of $\mu_1$ and $\mu_3$ in $\mu_1 \cdot \gamma'$ are equal. From our formula for $\mu_1 \cdot \gamma'$, this gives
\begin{displaymath}
\tfrac{1}{4} (2-\kappa-\kappa^{-1}) = \tfrac{1}{2} (1+\kappa^{-1}),
\end{displaymath}
which simplifies to $\kappa^2=-3$.
\end{proof}

It is not clear to us if Case~IIb can actually occur. The above proposition determines what a category with such an object would look like in some low degrees.

\end{document}